\documentclass[twoside]{article}

\usepackage[accepted]{aistats2022}
\usepackage{float, amssymb, amsmath, amsthm, graphicx, bbm, multirow, siunitx, placeins}
\usepackage[dvipsnames]{xcolor}
\usepackage{enumitem}
\setlist{nolistsep}

\usepackage{algorithm}
\usepackage{algpseudocode}
\algnewcommand\algorithmicinput{\textbf{INPUT:}}
\algnewcommand\INPUT{\item[\algorithmicinput]}
\algnewcommand\algorithmicoutput{\textbf{OUTPUT:}}
\algnewcommand\OUTPUT{\item[\algorithmicoutput]}

\usepackage{hyperref}[]
\usepackage{cleveref}

\newtheorem{theorem}{Theorem}
\newtheorem{lemma}[theorem]{Lemma}
\newtheorem{proposition}[theorem]{Proposition}
\newtheorem{corollary}[theorem]{Corollary}

\newtheorem{definition}{Definition}
\newtheorem{remark}{Remark}

\newtheorem{model}{Model}

\allowdisplaybreaks

\DeclareMathOperator*{\argmin}{arg\,min}
\DeclareMathOperator*{\argmax}{arg\,max}

\usepackage[round]{natbib}

\usepackage{mathtools}

\def\P{\mathbb{P}}

\def\th{\mathrm{th}}

\def\R{\mathbb{R}}

\begin{document}

\runningtitle{Denoising and Change Point Localisation}
\runningauthor{Wang, Madrid Padila, Yu, Rinaldo}
    
%

%

\twocolumn[

\aistatstitle{Denoising and Change Point Localisation in Piecewise-Constant High-Dimensional Regression Coefficients}

\aistatsauthor{ Fan Wang \And Oscar Hernan Madrid Padilla}

\aistatsaddress{Department of Statistics\\University of Warwick \And  Department of Statistics\\ University of California, Los Angeles } 

\aistatsauthor{Yi Yu \And Alessandro Rinaldo}
\aistatsaddress{Department of Statistics\\University of Warwick \And Department of Statistics and Data Science\\Carnegie Mellon University} 

]

\begin{abstract}
 We study the theoretical properties of the fused lasso procedure originally proposed by \cite{tibshirani2005sparsity} in the context of a linear regression model in which the regression coefficient are totally ordered and assumed to be sparse and piecewise constant.  Despite its popularity, to the best of our knowledge, estimation error bounds in high-dimensional settings have only been obtained for the simple case in which the design matrix is the identity matrix.  We formulate a novel restricted isometry condition on the design matrix that is tailored to the fused lasso estimator and derive estimation bounds for both the constrained version of the fused lasso assuming dense coefficients and for its penalised version.  We observe that the estimation error can be dominated by either the lasso or the fused lasso rate, depending on whether the number of non-zero coefficient is larger than the number of piecewise constant segments.  Finally, we devise a post-processing procedure to recover the piecewise-constant pattern of the coefficients. Extensive numerical experiments support our theoretical findings.
\end{abstract}

\section{INTRODUCTION}\label{sec-introduction}

High-dimensional linear regression models have been at the centre of statistical research and applications for decades due to advances in the  data collection and storage technologies in numerous application areas, including genetics \citep[e.g.][]{guo2019optimal}, finance \citep[e.g.][]{pun2016robust}, economics \citep[e.g.][]{fan2011sparse}, cybersecurity \citep[e.g.][]{peng2018modeling} and climatology \citep[e.g.][]{li2020far}, among many others.  To be specific, in this paper, we consider the following model
\begin{equation}\label{eq1}
	y_i = a_i^{\top} x^* + \epsilon_i, \quad i = 1, \ldots, n,
\end{equation}
where $x^* \in \mathbb{R}^p$ is the unknown regression coefficient of interest, $\{\epsilon_i\}_{i = 1}^n$ is a sequence of independent mean-zero random variables, $\{a_i\}_{i = 1}^n \subset \mathbb{R}^p$ are the $p$-dimensional covariates, $\{y_i\}_{i = 1}^n \subset \mathbb{R}$ are the responses and the dimensionality $p$ is allowed to be a function of the sample size $n$. 

The model \eqref{eq1} has been extensively studied in a vast body of methodological and theoretical work.  To overcome the high dimensionality, a stream of penalisation-based methods have been developed based on the entrywise sparsity assumption that only a few entries of $x^*$ are non-zero.  This line of attack includes ridge estimators \citep[e.g.][]{hoerl1970ridge}, lasso estimators \citep[e.g.][]{tibshirani1996regression}, bridge estimators \citep[e.g.][]{frank1993statistical}, elastic net estimators \citep[e.g.][]{zou2006sparse}, to name a few, theoretical properties of which have been extensively studied \citep[e.g.][]{buhlmann2011statistics}.  Beyond the entrywise sparsity, different structural sparsity assumptions have also been exploited in the literature, e.g.~known group structures \citep[e.g.][]{yuan2006model, simon2013sparse}, low-rank structures \citep[e.g.][]{candes2009exact}, piecewise-polynomial patterns \citep[e.g.][]{kim2009ell_1,tibshirani2014adaptive}, among others.  Surrounding~\eqref{eq1}, various statistical tasks are of interest, including testing, estimation, prediction and recovery of specific sparsity patterns. We refer the readers to \cite{friedman2001elements} and \cite{buhlmann2011statistics} as comprehensive textbooks.

In this paper, for the high-dimensional linear regression problem \eqref{eq1}, we are concerned with estimating~$x^*$ with piecewise-constant assumptions, that there exists an unknown set of locations, namely change points, $S = \{t_1, \ldots, t_s\} \subset \{1, \ldots, p-1\}$ such that $x^*_i \neq x_{i+1}^*$ if and only if $i \in S$.
    Let $y = (y_1, \ldots, y_n)^{\top}$, $A = (a_1, \ldots, a_n)^{\top}$, $s = |S|$, $\hat{x}$ be an estimator of $x^*$ and $S(\hat{x})$ be the set of change points induced by $\hat{x}$ (see \Cref{sec-notation}).  We aim for both (i) $\|\hat{x} - x^*\|$ and (ii) $d_{\mathrm{H}}(S(\hat{x}), S)$ to be as small as possible, where $\|\cdot\|$ is the $\ell_2$-norm of a vector and $d_{\mathrm{H}}(\cdot, \cdot)$ is the Hausdorff distance between two sets.

The piecewise-constant assumption enforces that $x^*$ lies in an unknown $(s+1)$-dimensional subspace of~$\mathbb{R}^p$.  This assumption can be seen as a relaxation of the entrywise sparsity assumption.  Let $\|\cdot\|_0$ be the $\ell_0$-norm of a vector and suppose that $\|x^*\|_0 = s_0$.  It then holds that $s \leq 2s_0$.  In practice, allowing for a very dense regression coefficients, the piecewise-constant assumption is more realistic than the entrywise sparsity assumption, and some examples of this are given in \Cref{sec:real}.

The statistical tasks (i) and (ii) above are referred to as denoising and change point localisation, which are closely-related but substantially different.  For example, when the design $A$ is an identity matrix, the fused lasso estimator is optimal in denoising \citep{guntuboyina2020adaptive, ortelli2019prediction} but sub-optimal in localising change points \citep{lin2017sharp, zhang2019element}.  As for the special case of the design~$A$ being an identity matrix, \eqref{eq1} becomes the classical univariate change point problem \citep[e.g.][]{wang2020univariate, verzelen2020optimal}.  Another related problem is the high-dimensional linear regression change point detection problem \citep[e.g.][]{rinaldo2021localizing}, which is however different, in the sense that each observation is associated with a regression vector $x^*_t$ such that $x_t^* \neq x_{t+1}^*$ when $t$ is a change point.  

\cite{xu2019iterative}, arguably, is the closest-related paper to ours.  They assumed a known general graph associated with nodes $\{1, \ldots, p\}$ and edges $E \subset \{1, \ldots, p\}^{\otimes 2}$.  They studied an approximate estimator of the $\ell_0$-penalised optimisation problem 
\begin{equation} \label{eqn:l0}
  \mathrm{min}_{x \in \mathbb{R}^p} \left\{\|y - Ax\|^2 + \lambda \sum_{(i, j) \in E} \mathbf{1}\{x_i \neq x_j\}\right\}.
\end{equation}
Our setup can be seen as a special case of theirs, with the known graph being a chain graph, that is to say $E = \{(i, i+1), \, i = 1, \ldots, p-1\}$.  \cite{xu2019iterative} derived denoising properties of an approximate estimator, under a specific type of the restricted eigenvalue condition, namely the cut-restricted isometry property (c-RIP).  

In terms of the theoretical performances, the $\ell_0$-type penalties and constraints are expected to be superior than their $\ell_1$ counterparts, but suffer computationally, which is why \cite{xu2019iterative} studied an approximate solution to the NP-hard problem \eqref{eqn:l0}.  In this paper, we study the denoising and change point localisation properties of both penalised and constrained least squares estimators, under a range of scenarios.  The theoretical and numerical performances are comparable to those in \cite{xu2019iterative}, with a notably superior performance in real data analysis.  The contributions of this paper are summarised as follows.

\begin{itemize}[leftmargin=*]
    \item We provide an estimation error bound on an $\ell_1$-constrained least squares estimator in \Cref{thm-RIC-bound}, based on a novel version of restricted isometry condition (RIC, \Cref{def1}).  To the best of our knowledge, this type of result has only been shown before for the case $A = I$ \citep[e.g.][]{ortelli2018total}.  \Cref{prop1} is a sufficient condition for the RIC, imposed directly on the design distribution.  
    \item In addition to constrained estimators, in \Cref{thm_penalise_RIC_bound}, we also provide an estimation error bound on a penalised estimator, which is the estimator studied in \cite{tibshirani2005sparsity}, with the presence of both the fused lasso and lasso penalties.  We also show a phase transition phenomenon alternating between the regimes dominated by the piecewise-constant and entrywise sparsities.  We accompany \Cref{thm_penalise_RIC_bound} with in-depth discussions involving other additive types of penalties.
    \item In terms of recovering the piecewise-constant patterns, knowing that $\ell_1$-type estimators are sub-optimal, in \Cref{sec-post-processing} we provide multiple post-processing procedures to prompt the change point localisation consistency.  The post-processing procedures are of interest \emph{per se} for general post-processing change point estimators.  Extensive numerical results are provided to show the favourable performances of our estimators.
\end{itemize}

\subsection{Notation and Problem Setup} \label{sec-notation}

For any set $M$, we let $|M|$ denote its cardinality.  For any vector $v \in \mathbb{R}^p$, we let $S(v) = \{i: \, v_i \neq v_{i+1}\} \subset \{1, \ldots, p-1\}$ be the set of change points induced by $v$.  For any matrix $M$, let $\sigma _{\min}(M)$ and $\sigma _{\max}(M)$ be the smallest and largest singular values of $M$, respectively; let $\|M\|_{\mathrm{op}}$ and $\|M\|_{\infty}$ be the $\ell_2\to\ell_2$ operator norm and the entrywise supremum norm of $M$, respectively.  For any vector $v$, let $\|v\|_1$ and $\|v\|_{\infty}$ be the $\ell_1$- and entrywise supremum norms of $v$, respectively.  Let $D \in \mathbb{R}^{(p-1) \times p}$ be the difference operator that satisfies $D_{ij} = \mathbf{1}\{i = j\} - \mathbf{1}\{j-i = 1\}$, $i \in \{1, \ldots, p-1\}$, $j \in \{1, \ldots, p\}$. We refer to the quantity $\|Dv\|_1$ as the total variation of $v \in \mathbb{R}^p$. For any two sets $M_1, M_2 \subset\{1, \ldots, p\}$, let $d_{\mathrm{H}}(M_1, M_2) = \max\{d(M_1 | M_2), d(M_2 | M_1))\}$ be the Hausdorff distance between $M_1$ and $M_2$, where $d(M_1|M_2) = \max_{m_2 \in M_2} \min_{m_1 \in M_1} |m_1 - m_2|$ is the one-sided Hausdorff distance.

We formalise below the model assumption which is used throughout the paper.

\begin{model} \label{assume-model}
Given data $\{a_i, y_i\}_{i = 1}^n \subset \mathbb{R}^p \times \mathbb{R}$ satisfying~\eqref{eq1}, with $S(x^*) = \{t_1, \ldots, t_s\}$, $t_0 = 0$ and $t_{s+1} = p$, let $\Delta_{p, n} = \min_{i = 0,\ldots,s} (t_{i+1} - t_i)$ and $\kappa_{p, n} = \min_{i = 1,\ldots,s} \kappa_i = \min_{i = 1,\ldots,s} |x^*_{t_i+1} - x^*_{t_i}|$ be the minimal spacing and minimal jump size, respectively.
\end{model}

\section{DENOISING} \label{sec-method}

\subsection{A Constrained Estimator}\label{constrained_estimator}

Under \Cref{assume-model}, we first consider a constrained estimator
    \begin{equation}\label{eq4}
        \hat{x} = \hat{x}(V) = \argmin_{x \in \mathbb{R}^p} \left\{\| y- Ax\|^2:\, \|Dx\|_1\leq  V \right\},
    \end{equation}  
    where $V \geq 0$ is a tuning parameter.  The total variation constraint that $\|Dx\|_1 \leq V$ is imposed in response to the piecewise-constant pattern of $x^*$ that $\|Dx^*\|_0 = s$.  Although total variation constrained estimators are heavily exploited in the existing literature \citep[e.g.][]{rudin1992nonlinear, steidl2006splines, guntuboyina2020adaptive}, to the best of our knowledge, this is the first time total variation constrained estimators are theoretically studied with a general design matrix.
    
\begin{remark}[Constrained and penalised estimators]
An alternative to \eqref{eq4} is a penalised estimator, which is a solution to the optimisation problem
    \[
        \mathrm{min}_{x \in \mathbb{R}^p}\{\|y - Ax\|^2 + \lambda \|Dx\|_1\},
    \]
    where $\lambda > 0$ is a tuning parameter.  When the design matrix $A = I$, this penalised estimator is the fused lasso estimator and has been studied extensively \citep[e.g.][]{ortelli2019prediction, lin2017sharp}.  For a general design matrix $A$, the computational aspects of the penalised optimisation are studied in \cite{tibshirani2011solution}, while its theoretical properties are yet to be understood.  We focus on the constrained estimator \eqref{eq4} in this section, based on the knowledge that constrained estimators can be strictly better in denoising than their penalised counterparts, see \cite{guntuboyina2020adaptive} for a thorough discussion.
\end{remark}    

In order to ensure good theoretical performances of~\eqref{eq4}, we assume the following condition.

\begin{definition} [Restricted isometry condition, RIC] \label{def1}
Let $\eta_1, \eta_2 > 0$ and $\rho_1, \rho_2:\, [0, \infty) \to [0,\infty)$ be increasing functions. A matrix $A \in \mathbb{R}^{n \times p}$ satisfies the $(\eta_1, \eta_2, \rho_1, \rho_2, t)$-restricted isometry condition (RIC) if  
	\begin{equation}\label{eq-RIC-statement-def}
       1 - \eta_1 - \sqrt{\rho_1(t)} \leq \|Ax\| \leq 1 + \eta_2 + \sqrt{\rho_2(t)}, 	    
	\end{equation}
	for all 
	\begin{equation}\label{eq-RIC-X-def}
	    x \in \mathcal{X}(t) = \{v \in \mathbb{R}^p: \, \|Dv\|_1 \leq t, \, \|v\|= 1\}.
	\end{equation}
\end{definition}

\Cref{def1} can be seen as an $\ell_1$-version of the c-RIP condition proposed and justified in \cite{xu2019iterative}.  More discussions on \Cref{def1} are available in \Cref{sec-RIC}. Before that, we present an upper bound on the estimation error of $\hat{x}$ based on \Cref{def1}.

\begin{theorem}\label{thm-RIC-bound}
Assume that the data are from \Cref{assume-model}, the minimal spacing condition $\Delta_{p, n} \geq c_0 p/(s+1)$ holds for some absolute constant $c_0 > 0$, and the design matrix $A$ satisfies $(\eta_1, \eta_2, \rho_1, \rho_2, V)$-RIC, where 
\[
    V \geq V^* = \|Dx^*\|_1 \quad \mbox{and} \quad 1 - \eta_1 - \rho_1(2V) > c_1,
\]
for some absolute constant $c_1 > 0$.  Let $\epsilon = (\epsilon_1, \ldots,$ $\epsilon_n)^{\top} \in \mathbb{R}^n$ be the noise vector.  Let $\hat{x}$ be the minimiser defined in \eqref{eq4} with the tuning parameter $V$.

If $\epsilon \sim \mathcal{N}(0, \sigma^2 I_n)$, then there exists a positive constant $C > 0$, such that it holds with probability at least $1 - 2(p \vee n)^{-1}$,
    \begin{align}\label{eq-thm-RIC-bound-2}
        & \|\hat{x}  - x^*\|^2 \leq C\sigma^2 s \log(p \vee n) \log \{p/(s+1)\} \nonumber \\
        & \hspace{0.5cm} +  C\sigma^2(V -  V^*)^2 p/(s+1) \log(p \vee n).     
    \end{align}
\end{theorem}

\Cref{thm-RIC-bound} is the first theoretical result on fused lasso regularisation, under a piece-wise constant model, with a general design matrix, while the only previously related results -- \cite{guntuboyina2020adaptive} and \cite{ortelli2018total} -- merely concerned with the identity design matrix.  Although allowing for more general design matrices, comparing our result to theirs, we are only off by a logarithmic factor. 

We remark that without the condition $\Delta_{p, n} \geq c_0 p/(s+1)$, following from almost identical proofs, \Cref{thm-RIC-bound} still holds with the upper bound in \eqref{eq-thm-RIC-bound-2} inflated by $\Delta_{\max}/\Delta_{p, n}$, where $\Delta_{\max} = \max_{i=0,\ldots,s} (t_{i+1} -t_i)$.

As for the tuning parameter, we highlight the condition $V\geq V^*$ and acknowledge that we lack theoretical controls when $V < V^*$.  This goes in line with \cite{guntuboyina2020adaptive}, who also required $V\geq V^*$  for the constrained fused lasso estimator with identity design matrix.  As in \cite{guntuboyina2020adaptive}, the smaller $(V-V^*)^2$ the sharper the upper bound.  In practice, one can choose $V$ using information-type criteria \citep[e.g.][]{tibshirani2012degrees}.

As for the $\ell_0$-penalised counterparts, two sets of theoretical results are provided in \cite{xu2019iterative}: One (Theorem~3.4 therein) relies on a version of the restricted isometry condition on design matrices, but the model is assumed to be noiseless besides the design matrix; this setup is incomparable with the additive noise setting~\eqref{eq1} considered in this paper.  The other (Theorem~3.5 therein) is a noisy version result and the upper bound thereof is at least of the order of $\|\epsilon\|$.  When $\epsilon \sim \mathcal{N}(0, \sigma^2 I_n)$, it holds that  $\|\epsilon\| \gtrsim \sqrt{n}$ with high probability, which implies that the upper bound \eqref{eq-thm-RIC-bound-2} is sharper than that in \cite{xu2019iterative}.

\subsection{The Restricted Isometry Condition} \label{sec-RIC}

The estimation error bound in \Cref{thm-RIC-bound} holds under \Cref{def1}, which is a version of compatibility conditions.  In order to provide theoretical guarantees for penalised/constrained estimators in high-dimensional problems, compatibility conditions are often assumed.  In particular, \cite{van2018tight} showed that such conditions are necessary for analysing lasso estimators.

\Cref{def1} can be regarded as an $\ell_1$ counterpart of the c-RIP condition introduced in \cite{xu2019iterative}. Instead of restricting in $\mathcal{X}(t)$ defined in \eqref{eq-RIC-X-def}, \cite{xu2019iterative} considered the constraint $\{v \in \mathbb{R}^p: \, 1 \leq \|Dv\|_0 \leq t\}$.  For both our results and those in \cite{xu2019iterative}, the RICs are most useful when assuming $\eta_j$, $\rho(s) = O(1)$, $j = 1, 2$.  This in essence implies that constraining to $\mathcal{X}(t)$, the restricted eigenvalues of $A^{\top}A$ are all around one, i.e.~for any $x \in \mathcal{X}(t)$, it holds that 
\[
\eta_1 + \sqrt{\rho_1(t)} \leq \sqrt{x^{\top}(A^{\top}A - I_p)x} \leq \eta_2 + \sqrt{\rho_2(t)}.
\]
This is in stark contrast to most of the classical lasso estimation literature with entrywise sparsity assumption, where it is often assumed that $\sigma_{\min}\{\mathbb{E}(A^{\top}A)\}$ is of order $O(n)$.  This difference in assumptions directly results in the difference in the final estimation error rates.  To be specific, with some abuse of notation, assuming $\|x^*\|_0 = s$, it holds with large probability that the lasso estimator $\hat{x}^{\mathrm{lasso}}$ satisfies that $\|\hat{x}^{\mathrm{lasso}} - x^*\|^2 \lesssim s\sigma^2\log(p)/n$ \citep[e.g.][]{buhlmann2011statistics}.  However, despite the fact that our piecewise-constant assumption can be seen as a generalisation of the entrywise-sparsity assumption, \eqref{eq-thm-RIC-bound-2} reads as $\|\hat{x} - x^*\|^2 \lesssim s\sigma^2 \log(p)$.  We remark that such difference is merely due to different scalings on the designs - this is more prominent in the statement in \Cref{prop1} below.  We opt for the same scaling as that in \cite{xu2019iterative} for an easier comparison.

In spite of the aforementioned similarity between \Cref{def1} and the c-RIP in \cite{xu2019iterative}, replacing the $\ell_0$ constraint with its~$\ell_1$ counterpart leads to two fundamental differences.  Firstly, constraining the $\ell_1$-norm of $Dx^*$ is, generally speaking, weaker than constraining its $\ell_0$-norm, and therefore enriches the flexibility.  Secondly, since for any vector $v$, $\|Dv\|_0$ is invariant with respect to scaling $v$, but $\|Dv\|_1$ is not, involving an upper bound on $\|Dv\|_1$ in \eqref{eq-RIC-statement-def} makes it more difficult in the proofs.

We provide a sufficient condition for \Cref{def1}.

\begin{proposition} \label{prop1}
Let $A \in \R^{n \times p}$ be a matrix with rows $n^{-1/2} a_i$, $i \in \{1, \ldots, n\}$, where $\{a_i\}_{i = 1}^n$ are independent and identically distributed sub-Gaussian vectors, satisfying that $\|a_1\|_{\psi_2} \leq U$ and $\mathrm{Cov}(a_1) = \Sigma$, with $\sigma_{\min}(\Sigma) = (1 - \zeta)^2$, $\sigma_{\max}(\Sigma) = (1 + \zeta)^2$ and $\zeta \in (0, 1)$.  Assume that there exist absolute constants $c_1, c_2 > 0$ such that 
    \begin{equation}\label{eq5}
        c_1 \sqrt{n} \phi_n \leq \{ t \sqrt{p} + \log (p)\} \{\sqrt{p}+ \sqrt{\log(n)}\} \leq c_2 n^{3/2},
    \end{equation}
    for some diverging sequence $\phi_n$, with $\phi_n U^{-4} \to \infty$, as~$n$ grows unbounded.  For any $x \in \mathcal{X}(t)$ defined in \eqref{eq-RIC-X-def}, it holds with probability at least 
    \[
        1 - 2n^{-c_3} - 2\exp (-c_4 U^{-4} \phi_n),
    \]
    where $c_3$, $c_4 > 0$  are absolute constants, that
    \[
       1 - \zeta - 2^{-1/2}(1 - \zeta)  \leq \|Ax\| 
       \leq 1 + \zeta + 2^{-1/2}(1 - \zeta).
    \]
\end{proposition}

We readily see that with large probability, under the conditions in \Cref{prop1}, the design matrix $A$ satisfies the $\left(\zeta, \zeta, 2^{-1}(1 - \zeta)^2, 2^{-1}(1 - \zeta)^2, t\right)$-RIC.  The condition \eqref{eq5} implies that the dimensionality $p$ can at most grow polynomially with respect to the sample size $n$, which is stronger than allowing for an exponential growth, as in \cite{xu2019iterative}.  Condition \eqref{eq5} is a direct consequence of the metric entropy of the constrain set $\mathcal{X}(t)$.  Specifically, as shown in \cite{guntuboyina2020adaptive}, for any $\delta > 0$, the $\delta$-metric entropy of $\mathcal{X}(t)$ is upper bounded by $\delta^{-1}\{t\sqrt{p} + \log(p)\}$, which is polynomial in $p$. Due to this bottleneck, we conjecture that the condition \eqref{eq5} cannot be improved, if one trades $\ell_0$-sparsity with the weaker version $\ell_1$-sparsity.

An immediate consequence of combining \Cref{thm-RIC-bound} and \Cref{prop1} is as follows.

\begin{corollary}\label{cor-prop+thm}
Assume that the data are from \Cref{assume-model}, the minimal spacing condition $\Delta_{p, n} \geq c_0 p/(s+1)$ holds for some absolute constant $c_0 > 0$, and the design matrix $A$ satisfies the conditions in \Cref{prop1}.  Under all the assumptions in \Cref{prop1}, let 
\[
    V \geq V^* = \|Dx^*\|_1 \quad \mbox{and} \quad 1 - \zeta - \rho_1(2V) > c,
\]
for some absolute constant $c > 0$.  Let $\epsilon = (\epsilon_1, \ldots, \epsilon_n)^{\top} \in \mathbb{R}^n$ be the noise vector.  

If $\epsilon \sim \mathcal{N}(0, \sigma^2 I_n)$, then there exists a positive constant $C > 0$, such that with probability at least 
\[
    1 - 2(p \vee n)^{-1} - 2n^{-c_3} -2\exp (-c_4 U^{-4} \phi_n ),
\]
it holds 
    \begin{align*}
        & \|\hat{x}  - x^*\|^2 \leq C\sigma^2 s \log(p \vee n) \log \{p/(s+1)\} \\
        & \hspace{0.5cm} +  C\sigma^2(V -  V^*)^2 p/(s+1) \log(p \vee n),
    \end{align*}
    where $c_3$, $c_4 > 0$ are absolute constants.
\end{corollary}

\subsection{A Penalised Estimator}\label{sec-penalise-theory}

Recall that in 
\cite{tibshirani2005sparsity}  the fused lasso estimator was introduced as 
    \begin{equation}\label{eq-fused-lasso-def}
        \tilde{x} = \argmin_{x \in \mathbb{R}^p} \left\{\|y - Ax\|^2 + \lambda_1 \|x\|_1 + \lambda_2 \|Dx\|_1\right\},
    \end{equation}
    where $\lambda_1, \lambda_2 > 0$ are tuning parameters.  The estimator $\tilde{x}$ is rooted in the belief that $x^*$ is both piecewise-constant and entrywise-sparse.  In this section, we are to establish theoretical guarantees on  the denoising performance of $\tilde{x}$, which, to the best of our knowledge, is not seen in the existing literature.
   
\begin{theorem}\label{thm_penalise_RIC_bound}
Assume that the data are from \Cref{assume-model}.  Let $H = \{j: \, x^*_j \neq 0\} \subset \{1, \ldots, p\}$, with $|H| = h$.  Assume that the design matrix $A$ has rows $\{n^{-1/2} a_i\}_{i = 1}^n$, where $\{a_i\}_{i = 1}^n$ are independent and identically distributed as $\mathcal{N}(0, \Sigma)$, with $\sigma_{\min}(\Sigma) \geq l > 0$ and $\|\Sigma\|_{\infty} \leq \rho$, where $l$ and $\rho$ are absolute constants.  Suppose that $\epsilon \sim \mathcal{N}(0, \sigma^2 I_n)$ and $(s + h) \log(p) \leq C_1 n$, where $C_1 > 0$ is a sufficiently large constant.  
    
Let $\tilde{x}$ be defined in \eqref{eq-fused-lasso-def} with $\lambda_1 = C_{\lambda_1} \sigma\sqrt{\rho\log (p \vee n)}$ and $\lambda_2 = C_{\lambda_2} \sigma\sqrt{\log (p \vee n)}$, where $C_{\lambda_1}, C_{\lambda_2} > 0$ are sufficiently large absolute constants.  It holds, with probability at least $1 - c_1 (p \vee n)^{-1}- p\exp(-n) - c_2\exp(-c_3n)$, that 
    \begin{equation}\label{eq-fused-lasso-thm-result}
        \|\tilde{x} - x^*\|^2 \leq  C  \sigma^2 \log(p \vee n) (s \vee h),
    \end{equation}
    where $C > 0$ is an absolute constant.
\end{theorem}

We remark that, \Cref{thm_penalise_RIC_bound} also relies on some form of restricted eigenvalue conditions, although we do not assume it explicitly as what we do in \Cref{thm-RIC-bound}.  Due to the Gaussian design, we in fact show that a form of RIC holds with large probability and leads to the success of \Cref{thm_penalise_RIC_bound} \citep[see e.g.~Lemma 9 in][]{zhang2015change}. 

The result \eqref{eq-fused-lasso-thm-result}, in fact, reads as $\|\tilde{x} - x^*\|^2 \lesssim h\lambda_1^2 + s \lambda_2^2$, with the two terms in the upper bound corresponding to the optimal estimation error bounds in the cases solely under the entrywise sparsity and piecewise-constant assumptions, respectively.  \Cref{thm_penalise_RIC_bound} unveils a phase transition phenomenon that when $h \gtrsim s$, the lasso rate dominates and \emph{vice versa}. 

The estimator in \eqref{eq-fused-lasso-def} is within the category of doubly-penalised estimators.  The additive type upper bound shown in \Cref{thm_penalise_RIC_bound} echos the general phenomenon for doubly-penalised estimators.  For instance, in high-dimensional linear regression problems, \cite{hebiri2011smooth} studied an estimator penalised by both $\ell_1$- and $\ell_2$-penalties, which is a generalisation of the elastic net estimator \citep[e.g.][]{zou2006sparse}, and derived an upper bound in the form of the sum of that on lasso and ridge estimators.  In the high-dimensional functional data analysis literature, with potentially many different functional covariates, \cite{wang2020functional} showed that the prediction upper bound is the sum of an upper bound related with the smoothness penalty and an upper bound related with the high-dimensionality.

\section{CHANGE POINT LOCALISATION AND POST-PROCESSING PROCEDURES} \label{sec-post-processing}

In high-dimensional linear regression problems, in addition to controlling the estimation error of estimators as in \Cref{sec-penalise-theory}, when it is believed that the regression coefficients are piecewise-constant, e.g.~there exist unknown group structures in the predictors, it is also of particular interest to recover such structures by localising change points.  For instance, in group lasso estimation, such group structures are assumed known.  An accurate estimation of the group structures can be a pre-processing step for a group lasso estimation.

In this section, the goal is to obtain a consistent change point estimator $\widehat{S}$, such that 
\begin{equation}\label{eq-consistency-def-est}
 d_{\mathrm{H}} (\widehat{S}, S) \leq \varepsilon_{p, n} \quad \mbox{and} \quad |\widehat{S}| = |S|,
\end{equation}
where $\varepsilon_{p, n}/p \to 0$, with probability tending to one as $n \to \infty $.  We refer to $\varepsilon_{p, n}$ as the localization error.  In view of the two criteria detailed in \eqref{eq-consistency-def-est}, the following two subsections focus on the guarantees on the localisation error and estimated number of change points, respectively.

\subsection{The Localisation Error}\label{mean filter}

When the design matrix $A$ is the identity, Theorem~4 of  \cite{lin2017sharp} shows that an upper bound on the estimation error  $\|\hat{x} - x^*\|$ immediately guarantees a control on the localisation error in terms of a one-sided Hausdorff distance $d \left(S \left( \hat x\right)| S \right)$.  Combining this result with  \Cref{thm-RIC-bound} (or \Cref{cor-prop+thm}), yields that, with probability at least $1 - 2( p \vee n)^{-1}$,
    \begin{align*}
        & d \left( S \left(\hat x\right) | S \right)  \lesssim \sigma^2 s \log ( p \vee n ) \log \{ p/(s+1) \}\kappa_{p,n}^{-2}\\
        & \hspace{0.5cm} + \sigma^2 (V -  V^*)^2 p/ (s+1) \log(p \vee n)\kappa_{p,n}^{-2}. 
    \end{align*}
Furthermore, \cite{lin2017sharp} point out that the fused lasso estimator $\hat{x}$, lacks a two-sided Hausdorff distance control because typically $|S(\hat{x})| > s$.  This issue with the over-estimation is partially tackled in \cite{lin2017sharp}.  There, the authors showed that a mean filtering post-processing procedure is effective at removing spurious estimated change points that are far away from the true change points. This post-processing step,  originally analysed by \cite{lin2017sharp} assuming an identity design matrix, can in fact be used for general design cases without modification. In detail, 
let $I_F(\hat x) \subset \{ b_{p,n}, \dots, p-b_{p,n}\}$ such that
    \begin{align}\label{eq-mean-filter-def-2}
        & I_F(\hat x) = \{i: \, i \in S(\hat x) \mbox{ or } i+b_{p,n} \in S(\hat x) \nonumber \\
        & \hspace{0.5cm} \mbox{or } i-b_{p,n} \in S(\hat x)\} \cup \{b_{p, n}, p-b_{p,n}\},
    \end{align}
    with an integer bandwidth $b_{p, n}>0$.  For each  $i \in \{ b_{p,n}, \dots, p-b_{p,n}\}$, set
    \begin{equation}\label{eq-mean-filter-def-1}
        F_i(\hat x) = \frac{1}{b_{p,n}} \sum_{j=i+1}^{i+b_{p,n}} \hat x_j - \frac{1}{b_{p,n}} \sum_{j=i-b_{p,n}+1}^{i} \hat x_j.
    \end{equation}
   The set of estimates change point is defined as the 
    \begin{equation}\label{eq-mean-filter-def-3}
        S_I(\hat x) = \{i \colon |F_i(\hat x)\vert \geq \tau_{p, n} \} \subset I_F(\hat x),
    \end{equation}
    where $\tau_{p,n} > 0$ is a pre-specified threshold.  An adaptation to Theorem~5 in \cite{lin2017sharp} leads to the following result.

\begin{proposition}\label{thm_two_side_hausdorff_distance}
Assume all the conditions in \Cref{thm-RIC-bound} hold.  Let 
\begin{align*}
    R_{p, n} = C \sigma^2 s \log(p \vee n) \log \{p/(s+1)\} \\
    + C \sigma^2 (V -  V^*)^2 p/(s+1) \log(p \vee n)
\end{align*}
and assume that there exists a sufficiently large absolute constant $C_{\mathrm{SNR}} > 0$ such that $\kappa^2_{p, n}\Delta_{p, n} > C_{\mathrm{SNR}} R_{p, n}$.
For $S_I(\hat{x})$ defined through \eqref{eq-mean-filter-def-2}, \eqref{eq-mean-filter-def-1}, and \eqref{eq-mean-filter-def-3}, with $\tau_{p,n}$ and $b_{p,n}$ satisfying that $\tau_{p, n} = c_{\tau} \kappa_{p, n}$, for some constant $c_{\tau} \in \left(0, 1 \right)$ and 
\[
b_{p, n} = c R_{p,n}\{c_{\tau} \wedge (1 - c_{\tau})\}^{-2}\kappa_{p, n}^{-2}  \leq \Delta_{p, n}/4,
\]
for a sufficiently large absolute constant $c > 0$, it holds with probability at least $1 - 2( p \vee n)^{-1}$ that, with a sufficiently large absolute constant $C_1 > 0$,
\begin{align*}
    & d_\mathrm{H} \left( S_I (\hat x), \, S \right) \leq C_1\sigma^2 s \log (  p \vee n ) \log \{ p/(s+1) \}\kappa_{p,n}^{-2} \\
    & \hspace{0.5cm} + C_1\sigma^2 (V -  V^*)^2 p/(s+1 ) \log(p \vee n)\kappa_{p,n}^{-2}. 
\end{align*}
\end{proposition} 

As we can see from \Cref{thm_two_side_hausdorff_distance}, the mean filter smooths the fused lasso estimator and delivers 
an estimator satisfying the first criterion in \eqref{eq-consistency-def-est}.
When $V = V^*$, we see that the localisation rate is, aside from  a poly-logarithmic factor, of order $\sigma^2 \kappa^{-2}_{p, n} s$.  This type of guarantee is consistent with analogous results established in  array of change point detection problems and implies the near-optimality of our results. Indeed, in change point detection problems, minimax lower bounds on the localisation errors are usually of the form of $\sigma^2 \kappa^{-2}$ and their upper bounds obtained by polynomial-time algorithms are usually in the form of $\sigma^2 \kappa^{-2} \times \mbox{a sparsity parameter} \times \mbox{a logarithmic factor}$, see \cite{yu2020review}.

\subsection{The Number of Change Points} \label{sec_number_of_change_points}
The output $S_I(\hat{x})$ is, as pointed out in \cite{lin2017sharp}, still an over-estimator of $S$, due to the lack of control on the estimated change points around the true change points.  To obtain a consistent change point estimator, we further prune the estimator with an additional time filter, detailed below.

Provided that $S_I(\hat{x}) \neq \emptyset$, we sort it as $\{\hat{t}_1, \ldots, \hat{t}_M\}$, with $\hat{t}_i < \hat{t}_j$, $i < j$.  If $M = 1$, we let $S_T(\hat{x}) = S_I(\hat{x})$.  If $M \geq 2$, we let $D_m = \hat{t}_{m+1} - \hat{t}_m$, $m \in \{1, \ldots, M - 1\}$.  With a pre-specified tuning parameter $t_{p, n} > 0$, let $\widetilde{S} = \{m:\, D_m > t_{p, n}\} = \{s_1, \ldots, s_{|\widetilde{S}|}\}$,   which partitions $S_I(\hat{x})$ into $|\widetilde{S}|+1$ segments, that is 
\begin{align*}
    S_I(\hat{x}) & = \{\hat{t}_1, \ldots, \hat{t}_{s_1}\} \cup \cdots \cup \{\hat{t}_{s_{|\widetilde{S}|} + 1}, \ldots, \hat{t}_M\} \\
    & = \cup_{u = 1}^{|\widetilde{S}|+1} \widetilde{S}_u.
\end{align*}
If $|\widetilde{S}| = 0$, then we have that $S_I(\hat{x}) = \widetilde{S}_1$.  Finally, denote that
\begin{equation}\label{eq-sthatx-def}
    S_T(\hat{x}) = \{\mathrm{median}(\widetilde{S}_u), \, u  = 1, \ldots, |\widetilde{S}|+1\},
\end{equation}
where we take the convention that the median of a set is a member in this set and when there are two different medians, we take the smaller one for uniqueness.  In the following we show that $S_T(\hat{x})$ provides a consistent estimation on the number of change points.

\begin{proposition}\label{thm_number_change_points}
Assume all the conditions in \Cref{thm_two_side_hausdorff_distance} hold.  Let $S_T(\hat{x})$ be defined in \eqref{eq-sthatx-def} with the tuning parameter $t_{p, n} = 2b_{p, n}$.  It holds with probability at least $1 - 2( p \vee n)^{-1}$ that $|S_T(\hat{x})| = |S|$.
\end{proposition} 

With a further post-processing step based upon the mean filter proposed in \cite{lin2017sharp} and with properly chosen tuning parameters, we have now met the second criterion in 
\eqref{eq-consistency-def-est}.  Note that $S_T(\hat{x}) \subset S_I(\hat{x})$.  We further have that $S_T(\hat{x})$ is consistent in terms of \eqref{eq-consistency-def-est}.  The theoretical tuning parameters are functions of unknown parameters.  We discuss how to choose the tuning parameters in practice in \Cref{sec-numerical}.  

\section{NUMERICAL EXPERIMENTS}\label{sec-numerical}

Our focus for now is to study the piecewise-constant and potentially very dense regression coefficients.  We have shown that a fused lasso (FL) estimator is enough for the denoising purpose, which implies good prediction ability.  To learn the piecewise-constant pattern, a mean filtering post-processing based on the fused lasso (FLMF) achieves a nearly-optimal localisation error; to accurately partition the coefficients into groups, a further step based on a time filter (FLMTF) is necessary.  In this section, we study the numerical performances of the aforementioned three estimators.

The FL used to obtain our numerical results, is a penalised estimator, not the constrained estimator we focus on in \Cref{constrained_estimator}. This choice was made out of convenience, due to the availability of the well-written R package genlasso \citep{genlasso}.  We acknowledge this inconsistency of the paper.  The FL estimators are solved with tuning parameter chosen from a 5-fold cross-validation. 

As instructed in \cite{lin2017sharp}, the tuning parameter $b_{p, n}$ in FLMF is chosen to be $\lfloor 0.25 \log^2(p) \rfloor$ and $\tau_{p, n}$ is chosen via a permutation-based algorithm.  For FLMTF, the additional tuning parameter is chosen to be $t_{p, n} = 2 \times \lfloor 0.25 \log^2(p)\rfloor$.  The main competitor is an $\ell_0$-penalised estimator via the ITerative Alpha Expansion (ITALE) algorithm \citep{xu2019iterative}, the tuning parameter therein is chosen through a $5$-fold cross-validation.

\subsection{Simulation Studies} \label{sec:sim}

Four types of design matrices are considered: \textbf{(1)} Identity matrices.  \textbf{(2)} Band matrices.  Let $A$ be $A_{ij} = 0$, if $|i - j| > h$, and $A_{ij} \stackrel{\mbox{i.i.d.}}{\sim} \mathcal{N}(0, 1)$, if $|i - j| \leq h$, where $h \in \mathbb{N}$.  \textbf{(3)} Gaussian random matrices with identity covariance matrices.  Each row of $A$ is i.i.d.~from $\mathcal{N}(0, I_p)$.  \textbf{(4)} Gaussian random matrices with band covariance matrices.  Each row of $A$ is i.i.d.~from $\mathcal{N}(0, \Sigma)$, where $\Sigma$ is a band matrix with bandwidth $h$.  Each type of design matrices includes three scenarios: \textbf{(1)} one change point; \textbf{(2)} nine equally-spaced change points; and \textbf{(3)} nine unequally-spaced change points.  All details are deferred to the supplementary materials.  

For each setting, we fix $p = 1000$ and let the noise be from $\mathcal{N}(0, \sigma^2 I_n)$.  Each combination of parameters is repeated $100$ times.  Each reported result is in the form of mean and standard deviation.  We vary the jump size (denoted as a constant multiplied by $\gamma$ in this section), sample size, variance $\sigma^2$ and other model specific parameters to demonstrate a wide range of situations.  All results can be found in the supplementary materials.  

In \Cref{Fig_simulation_band_matrix}, we only present a few examples to convey the key messages.  Representative results of cases on four different types of design matrices are shown in the four rows of \Cref{Fig_simulation_band_matrix}.  From left to right, the columns are the results on the mean squared errors reflecting the denoising performances, the Hausdorff distances and $|S_T(\hat{x}) - S|$ reflecting the change point estimation performances.  Generally speaking, although the $\ell_0$ estimators outperform $\ell_1$ estimators theoretically, across the board, our three estimators concerned have comparable numerical performances to ITALE.  We would like to highlight two observations.  Firstly, in terms of the change point estimation performances, we see a clear improvement from the two filters.  Secondly, in the second and fourth rows, which correspond to band matrices in the designs, our estimators uniformly outperform the competitor ITALE.  This suggests a possible regime where our RIC is more flexible than the c-RIP.

\begin{figure}[h] 
\vspace{.3in}
\centering 
\includegraphics[width=0.47\textwidth]{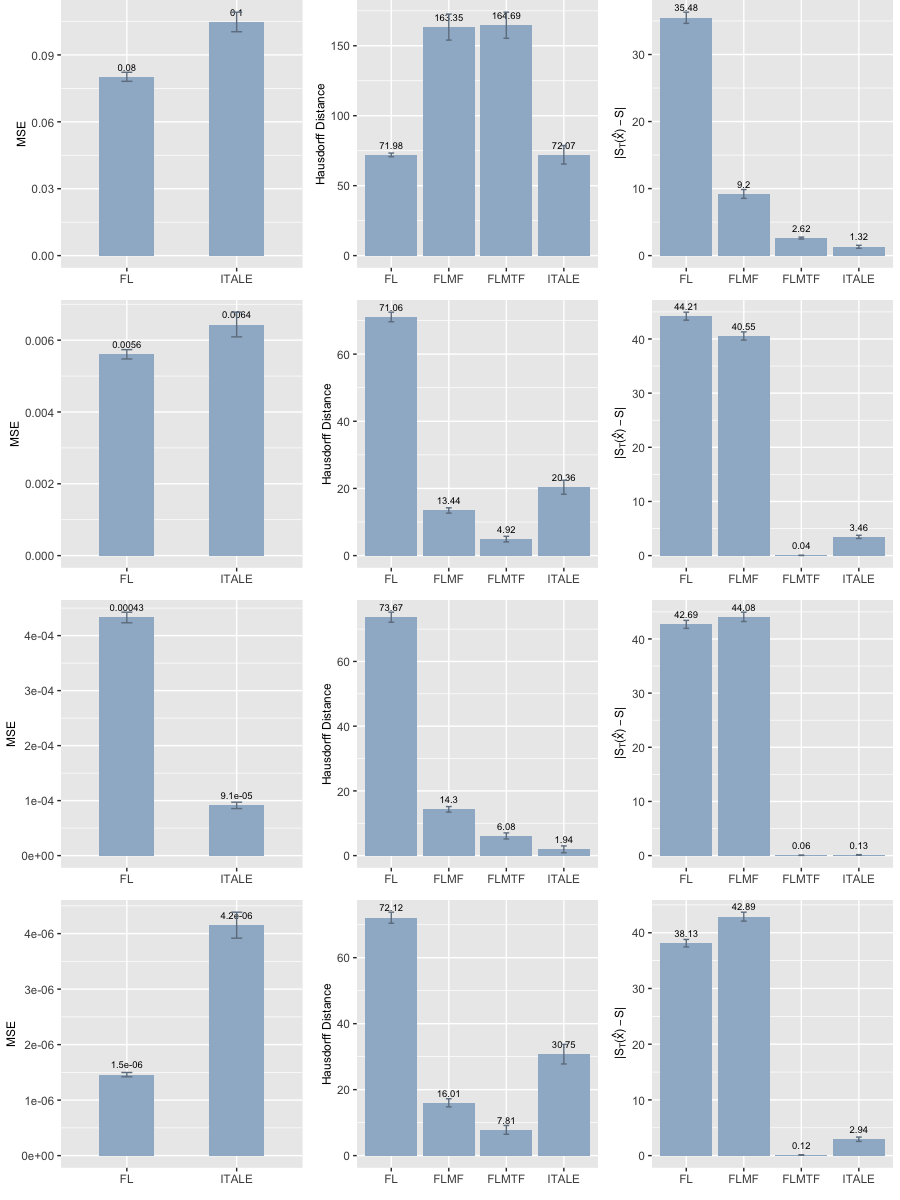}
\vspace{.3in}
\caption{From top to bottom: Identity matrix with nine equally spaced change points with  $\gamma = 1$, $\sigma=2$; Band design matrix with nine equally spaced change points with $\gamma = 1$, $\sigma=2$ and $h = 10$; Gaussian random design matrix with identity covariance matrix and nine equally spaced change points with  $\gamma = 1$, $\sigma=2$ and $n = 500$; Gaussian random design matrix with band covariance matrix and nine equally spaced change points with $n = 500$, $\gamma = 1$, $\sigma=2$ and $h = 50$.  From left to right: the mean squared errors; the Hausdorff distance and the number of estimated change points.  Each result is shown in the form of mean and standard error. FL: fused lasso; FLMF: fused lasso with mean filters; FLMTF: fused lasso with mean and time filters; and ITALE: \cite{xu2019iterative}. } \label{Fig_simulation_band_matrix} 
\end{figure}

\subsection{Real Data Analysis} \label{sec:real}

We consider two real data sets in this subsection, focusing on the denoising and change point localisation purposes, respectively.

\noindent \textbf{The cookie dough data set} \citep{osborne1984application, brown2001bayesian, hans2011elastic} contains $n = 72$ cookie dough samples, with the sucrose content as the response and the quantitative near-infrared (NIR) spectroscopy results corresponding to $p = 700$ equally-spaced wavelengths as the predictors.  Because of the finely-partitioned wavelengths predictors, we expect a piecewise-constant but non-sparse pattern of the regression coefficients.

We use this data set to demonstrate the denoising performances of FL.  As for the competitors, we include the classical lasso, the elastic net \citep{zou2006sparse} and ITALE.  The lasso and elastic net estimators are obtained by applying the R \citep{R} package glmnet \citep{glment}, with tuning parameters selected via a 5-fold cross-validation. 

We randomly split the original data set 100 times, each with a training data set of 39 observations and a test data set with 31 observations, after removing the 23rd and 61st observations as outliers \citep{hans2011elastic}.  Data are all centred to eliminate the intercepts, as required in the model assumption, and standardised before estimation to follow suit \citep[see e.g.~Lemma 9 in][]{hans2011elastic}.  The training data sets are used to obtain estimators and the test data sets are used to calculate the mean squared errors.

We collect the results in \Cref{Cookie_dough_data}, including MSE - mean squared errors,  $|\widehat{S}|$ - numbers of estimated change points and $|\widehat{H}|$ - numbers of non-zero estimated coefficients.  We do not include the results of ITALE, which returns an average MSE of 52840.  It can be seen from \Cref{Cookie_dough_data} that fused lasso allows for  very dense regression coefficients, but encourages recovering the piecewise-constant pattern, and most importantly, has the best prediction performances.

\begin{table}[h]
\caption{The results of different methods on the cookie dough data set.  MSE: mean squared errors; $|\widehat{S}|$: numbers of estimated change points; and $|\widehat{H}|$: number of non-zero estimated coefficients.  Each cell is in the form of mean(std.dev).}\label{Cookie_dough_data}
\begin{center}
\begin{tabular}{llll}
& \textbf{FL}  & \textbf{Lasso} & \textbf{Elastic net} \\ \hline
MSE & $0.122(0.091)$ & $0.136(0.097)$ &  $0.138(0.097)$ \\ 
$|\widehat{S}|$ & $10.8(2.40)$ & $35.2(7.21)$ & $330(204)$ \\ 
$|\widehat{H}|$ & $697(15.5)$ & $23.1(5.03)$ & $285(188)$\\  
\end{tabular}
\end{center}
\end{table}

\medskip  
\noindent \textbf{The air quality data set} \citep{airdata} consists of daily average air quality measurements from different cites around the world.  We choose $n = 30$ cities (Amsterdam, Bangkok, Beijing, Changsha, Chongqing, Dalian, Frankfurt, Guangzhou, Harbin, Hefei, Hong Kong, Kunming, Kyoto, Lhasa, London, Los Angeles, Manchester, Nanjing, Osaka, Paris, Sanya, Seoul, Shanghai, Singapore, Hamburg, Suzhou, Tianjin, Xiamen, Xi'an and Tokyo).  Let their Particulate Matter 2.5 ($\text{PM}_{2.5}$) values on $2021$-$07$-$30$ be the responses and let their weekly averaged $\text{PM}_{2.5}$ data from the previous $p = 52$ weeks as their own covariates.  Data are all centred and standardised.  Our goal is to detect potential change points in the coefficients consisting of 52 weeks.

We focus on the behaviours of FLMTF studied in \Cref{sec_number_of_change_points} and the ITALE studied in \cite{xu2019iterative}.  Similar observations of extremely large MSEs of ITALE lead us to discard discussing their results.  The FLMTF detects $4$ change points corresponding to the 6th, 21st, 33rd and 48th weeks, which are inline with the season changes, as depicted in \Cref{Fig_air_quality}.  Note that seasons are commonly defined in two ways \citep{trenberth1983seasons}: astronomical and meteorological seasons.  The start dates of astronomical seasons are 20th March, 20th July, 22nd September and 21st December; and those of the meteorological seasons are 1st March, 1st July, 1st September and 1st December.  In this specific data set, the astronomical seasons start in the 6th, 19th, 31st and 45th weeks, and the meteorological seasons start in the 9th, 22nd, 35th and 48th weeks.

\begin{figure}[h]
\vspace{.3in}
\centering 
\includegraphics[width=0.4\textwidth]{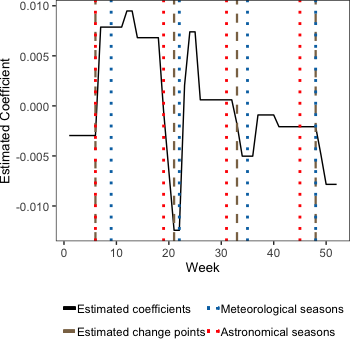}
\vspace{.3in}
\caption{The air quality data set.}\label{Fig_air_quality} 
\end{figure}

\section{Conclusions}

For a high-dimensional linear regression problem with piecewise-constant patterned coefficients, we derived the denoising performances of fused lasso estimators, in the form of both a constrained and penalised estimators.  A novel restricted isometry condition is proposed, with a sufficient condition.  The main results in this paper are only based on the fused lasso penalty, which encourages to recover the piecewise-constant pattern, but allows for a very dense regression coefficient in a high-dimensional setup.  We, however, also derived the denoising performances of the estimator originally proposed in \cite{tibshirani2005sparsity}, which assumes both fused lasso and lasso penalties.  This set of results unveils a phase transition phenomenon and echos the general results of doubly-penalised estimators.  The denoising performance is a guarantee for predictability.  To understand the piecewise-constant pattern and to consistently estimate the change points, we built upon \cite{lin2017sharp} and studied a doubly-filtered estimator, with the fused lasso estimator as the input.  Despite that the $\ell_0$-penalised estimators are theoretically superior than the $\ell_1$-type estimators studied in this paper, we show in some theoretical aspects, we derive sharper theoretical results than those in \cite{xu2019iterative} under some RICs, with a notably improvement in real data analysis.

\subsubsection*{Acknowledgements}
All the authors thanks to the reviewers for constructive comments.  Yu is partially funded by EPSRC EP/V013432/1.  Madrid Padilla and Rinaldo are partially funded by NSF DMS 2015489.

\clearpage
\bibliography{references}


\clearpage
\appendix

\thispagestyle{empty}

\onecolumn \makesupplementtitle

\section{ADDITIONAL NOTATION}

For any vector $v \in \mathbb{R}^p$ and any set $M \subset \{1, \ldots, p\}$, let $v_M = (v_i, i \in M)^{\top} \in \mathbb{R}^{|M|}$ and $v_{-M} = (v_i, i \notin M)^{\top} \in \mathbb{R}^{p - |M|}$ be sub-vectors of $v$.  For any matrix $Q \in \mathbb{R}^{p \times q}$ and any set $M \subset \{1, \ldots, p\}$, let $Q_M$ and $Q_{-M}$ be the submatrices of $Q$ only containing rows indexed by $M$ and $\{1, \ldots, p\} \setminus M$, respectively. For any subspace $W \subset \mathbb{R}^p$, let $P_W$ be the orthogonal projection onto $W$ and let $W^{\perp}$ be its orthogonal complement.  

Let $S^{\prime} \subset \{1, \ldots, p-1\}$ be any set. Let  $\mathcal{N}_{-S^{\prime}} = \{v \in \mathbb{R}^p:\, (Dv)_{-S^{\prime}} = 0\}$ and $r_{S^{\prime}} = \text{dim}(\mathcal{N}_{-S^{\prime}})$.  Let $\Psi^{-S^{\prime}} = (D_{-S^{\prime}})^{\top}\{D_{-S^{\prime}} (D_{-S^{\prime}})^{\top}\}^{-1}$.  The $j$th column of  $\Psi^{-S^{\prime}}$ is denoted as  $\psi^{-S^{\prime}}_j$.   For any vector $w_{-S^{\prime}}$ with $w_j \in [0, 1]$, $j \in \{1, \ldots, p-1\} \setminus S^{\prime}$ and any vector $v \in \mathbb{R}^p$, let $(1-w_{-S^{\prime}})(Dv)_{-S^{\prime}} = \{(1-w_j)(Dv)_j \}_{\{j  \in \{1, \ldots, p-1\}\backslash  S^{\prime}\}  }$.  

\begin{definition} [\cite{ortelli2019prediction}] \label{def2}
For any vector $q_{S^{\prime}} \in \{-1,1\}^{s^{\prime}}$, where $S^{\prime} \subset \{1, \ldots, p-1\}$ is any set and  $s^{\prime} = \left\vert S^{\prime} \right\vert$, the noiseless effective sparsity is defined as
	\[
	    \Gamma^2(q_{S^{\prime}} ) = \left(  \max\left\{   q_{S^{\prime}}^{\top}  (D\theta)_{S^{\prime}}   - \|(D\theta)_{-S^{\prime}}\|_1:\, \|\theta\| = p^{1/2}, \, \theta \in \mathbb{R}^p   \right\}\right)^2.
	\]
	The noisy effective sparsity is defined as
	\[
	\Gamma^2(q_{S^{\prime}},w_{-S^{\prime}} ) =  \left(  \max\left\{   q_{S^{\prime}}^{\top}  (D\theta)_{S^{\prime}}   - \|(1-w_{-S^{\prime}})(D\theta)_{-S^{\prime}}\|_1:~\|\theta\| = p^{1/2}, \, \theta \in \mathbb{R}^p   \right\}\right)^2,
	\]
	with
	\[
	w_j =  \frac{c_0 \left\|\psi^{-S^{\prime}}_{j} \right\|}{ \max_{l\in \{1, \ldots, p-1\} \backslash S^{\prime}} \left\|\psi^{-S^{\prime}}_{l} \right\|  }
	\]
	for  $j \in \{1, \ldots, p-1\} \backslash S^{\prime}$, and for a small enough constant $c_0$. 
\end{definition}

\section[]{PROOFS OF THE RESULTS IN SECTION \ref{sec-method}}
\subsection[]{Proof of \Cref{thm-RIC-bound}}

\begin{lemma}\label{lemma_s}
Suppose that $S$ satisfies the minimal spacing condition  $\Delta_{p, n} \geq c_0 p/(s+1)$ for some large enough constant $c_0 > 0$.
There exists a set $S^{\prime} = \left\{ u_1, \ldots, u_m \right\}$ with $0 = u_0 < u_1 < \ldots < u_m < u_{m+1} = p$, satisfying
\begin{enumerate}
    \item [(1)] $m \asymp s+1$,
    \item [(2)] $u_{\min} = \min_{t =  0, \ldots,  m} (u_{t+1}- u_{t}) \geq c_1 p/(s+1)$,
    \item [(3)] $ u_{\max} = \max_{t = 0, \ldots,  m} (u_{t+1}- u_{t}) \leq c_2 p/(s+1)$,
    \item [(4)] $S \subset S^{\prime}$,
\end{enumerate}
where $c_1, c_2 > 0$ are constants that only depend on $c_0$.
\end{lemma}
\begin{proof}
Let  $\delta =  c p/(s+1)$ for some small enough constant $c>0$,    and suppose that $G = \{v_1,\ldots,v_{m^{\prime}} \}$ is grid of evenly space points in $[1,p]\cap \mathbb{N}$ with spacing $\delta$. Let us set
\[
S^{\prime} =  S \cup \left\{v \in G:~\vert v - a\vert \geq \frac{cp}{2(s+1)},\,\,\forall  a \in S\right\}.
\]
Then the set $S^{\prime}$ satisfies the required. Indeed, notice that by definition of $G$ and $S^{\prime}$,  $ \vert S\vert \leq \vert S^{\prime} \vert \leq  \vert S\vert + \vert G \vert  \lesssim s+1$.  Also, $S \subset S^{\prime}$.

Next, let us verify $(2)$. To see this notice that if $u,v \in S$ and $u\neq v$, then 
\[
    |u-v| \geq \frac{c_0 p}{s+1}.
\]
If $u,v \in S^{\prime}  \backslash S$ and $u\neq v$, then  by construction of $G$,
\[
 |u-v| \geq \frac{c p}{s+1}.
\]
If  $u \in S$ and $v \in S^{\prime}  \backslash S$, then by definition of $S^{\prime}$ it holds that 
\[
 |u-v| \geq \frac{c p}{2(s+1)}.
\]
Therefore,
\[
   u_{\min} \geq \min\{c_0,c/2\}\frac{p}{s+1}.
\]

Finally, let us verify $(3)$. To see this, we proceed in cases.
\begin{enumerate}
    \item [(a)] Notice that  $u_j,u_{j+1} \in S^{\prime}$ cannot happen if $c \leq c_0$ with $c$ small enough. Otherwise any $a \in G \cap [u_j, u_{j+1}]$ would need to satisfy $ \min\{\vert a - u_{j} \vert,\vert a - u_{j+1} \vert\} \leq cp/(2(s+1))$. But this is not possible for small enough $c$.
   \item [(b)] \label{the_second} Suppose  that $u_{j}, u_{j+1} \in G$. Then by construction 
\[
\vert u_{j+1} - u_{j} \vert  = \frac{ cp }{s+1}.
\]
    \item [(c)] Suppose that $u_j \in S$ and $u_{j+1} \in S^{\prime} \backslash S$.   If  $\vert u_{j+1} - u_{j} \vert > 3 cp/(2(s+1))$, then there must exists $a \in G$ such that $u_{j} < a< u_{j+1}$ and
$a - u_{j} > cp/(2(s+1))$. Therefore, it must be the case that $a \in S^{\prime}$. But that it is not possible because it would mean that $u_{j} < a< u_{j+1}$ and $a \in S^{\prime}$. Hence, 
\[
\vert u_{j+1} - u_{j}\vert \leq  3 cp/(2(s+1)).
\]

   \item [(d)] Suppose that $u_{j+1} \in S$ and $u_{j} \in S^{\prime} \backslash S$. With the same argument as in (\ref{the_second}), we can conclude that 
\[
\vert u_j - u_{j+1}\vert \leq  3 cp/(2(s+1)).
\]
\end{enumerate}
The proof concludes.

\end{proof}
\begin{proof}[Proof of \Cref{thm-RIC-bound}]

Let $S^{\prime}$ the set constructed in \Cref{lemma_s}. Notice that  $\|D_{-S^{\prime}}x^*\|_1 =0 $ and $\left\vert S^{\prime} \right\vert \asymp s+1$. 

It directly follows from the definition of $\hat{x}$ that
\begin{align}\label{eqn1}
    \|A\hat{x} - Ax^*\|^2 \leq 2\epsilon^{\top}A(\hat{x}-x^*) \leq |2\epsilon^{\top}A  P_{\mathcal{N}_{-S^{\prime}}}(\hat{x}-x^*)| +  |2\epsilon^{\top}A  P_{\mathcal{N}_{-S^{\prime}}^{\perp}}(\hat{x}-x^*)| = |(I)| + |(II)|.
\end{align}
The rest of the proof will be conducted on the two terms $(I)$ and $(II)$ in the right-hand side of \eqref{eqn1}.

\medskip
\noindent \textbf{Term $(I)$.}  Let $\left\{v_1, \ldots, v_{r_{S^{\prime}}} \right\} \subset \R^p$ be an orthonormal basis of $\mathcal{N}_{-S^{\prime}}$, such that for the partition $I_1, \ldots,   I_{r_{S^{\prime}}}$ of $\{1, \ldots, p\}$ induced by the change point set $S$, it holds that 
\[
    v_{l, i} = \begin{cases}
        |I_l|^{-1/2} & i \in I_l,\\
        0 & \text{otherwise}.
    \end{cases}  
\] 
Based on this basis, we have that 
\begin{align}
    (I)^2 \leq & 4 \|P_{\mathcal{N}_{-S^{\prime}}} A^{\top} \epsilon\|^2  \|\hat{x} - x^*\|^2 = 4 \left\|\sum_{j=1}^{r_{S^{\prime}}} v_j^{\top} A^{\top} \epsilon v_j \right\|^2 \|\hat{x}-x^*\|^2 \nonumber \\
    \leq & 4 \|\hat{x} - x^*\|^2  r_{S^{\prime}} \left(\max_{j = 1, \ldots , r_{S^{\prime}}} \frac{|\epsilon^{\top} A v_j|^2}{\| Av_j  \|^2}\right) \left(\max_{j = 1, \ldots, r_{S^{\prime}}} \|Av_j \|^2 \right). \label{eq-pf-thm1-term-1}
\end{align}
Define the event
\[
    \Omega_1 = \left\{ \max_{j = 1, \ldots, r_{S^{\prime}}} \frac{|\epsilon^{\top} A v_j|^2}{\| Av_j \|^2 }  \leq  2\sigma^2 \left\{\log(2 r_{S^{\prime}})  +  \log \left( p \vee n \right) \right\} \right\}.
\]
Due to a concentration inequality on the maximum of $r_{S^{\prime}}$ standard Gaussian random variables \citep[e.g.~Lemma~17.5 in][]{van2016estimation}, we have that $\mathbb{P} \{\Omega_1 \} \geq  1- \left( p \vee n \right)^{-1} $.  On the event $\Omega_1$, due to \eqref{eq-pf-thm1-term-1}, we have that  
\begin{align}\label{eqn3}
    |(I)| & \leq 2\sqrt{2} \sigma  \sqrt{r_{S^{\prime}}     \left\{ \log(2 r_{S^{\prime}})  +  \log \left( p \vee n \right)  \right\}}  \|\hat{x}-x^*\| \max_{ j=1,\ldots,r_{S^{\prime}} } \| Av_j  \| \nonumber \\
    & \leq 2\sqrt{2} \sigma \left(1     +\eta _2+   \sqrt{\rho_2( 1 )}\right) \sqrt{r_{S^{\prime}}     \left\{ \log(2 r_{S^{\prime}})  +  \log \left( p \vee n \right)  \right\}} \|\hat{x}-x^*\|,
\end{align}
where the second inequality is due to the RIC condition in \Cref{def1} and $\|Dv_j \|_1 \leq 1$ for any $ j=1,\ldots,r_{S^{\prime}}$.

\medskip
\noindent \textbf{Term $(II)$.} Due to the choice of $V$, for any $\lambda > 0$, it holds that
\begin{align} \label{eqn6}
     (II)  & \leq  2\epsilon^{\top}A  P_{\mathcal{N}_{-S^{\prime}}^{\perp}}(\hat{x}-x^*) + \lambda (\|Dx^*\|_1 -  \|D\hat{x}\|_1) + \lambda (V -  V^*) \nonumber \\
    & =  2\epsilon^{\top}A \Psi^{-S^{\prime}} \{D(\hat{x}-x^*)\}_{-S^{\prime}} + \lambda (\|Dx^*\|_1 -  \|D\hat{x}\|_1) + \lambda(V -  V^*) \nonumber \\
    & \leq  \frac{2}{c_0}\left(\max_{j\in \{1, \ldots, p-1\} \backslash S^{\prime}} \frac{| \epsilon^{\top}A \psi_j^{-S^{\prime}}|}{\| A\psi_j^{-S^{\prime}}\|}\right) \left(\max_{j\in \{1, \ldots, p-1\} \backslash S^{\prime}} \frac{\| A\psi_j^{-S^{\prime}}\| }{\| \psi_j^{-S^{\prime}}\| }\right) \left(\max_{j\in  \{1, \ldots, p-1\}\backslash S^{\prime} } \| \psi_j^{-S^{\prime}} \|\right) \nonumber \\
    & \hspace{0.5cm} \|w_{-S^{\prime}} \{D(\hat{x}-x^*)\}_{-S^{\prime}} \|_1 + \lambda (\|Dx^*\|_1 -  \|D\hat{x}\|_1) + \lambda(V -  V^*),
\end{align}
where $c_0 > 0$ is an absolute constant.

Define the event
\[
    \Omega_2 = \left\{\max_{j\in \{1, \ldots, p-1\} \backslash S^{\prime}}\frac{ \left\vert  \epsilon^{\top}A \psi_j^{-S^{\prime}} \right\vert}{ \left\| A\psi_j^{-S^{\prime}} \right\|} \leq 2\sigma \sqrt{ \log \left( p \vee n \right)}  \right\}.
\]
It again follows from a concentration inequality on the maximum of $r_{S^{\prime}}$ standard Gaussian random variables \citep[e.g.~Lemma~17.5 in][]{van2016estimation} that $\mathbb{P}\{\Omega_2\} \geq 1 - ( p \vee n)^{-1}$.  

Recall that $S^{\prime} = \{u_1, \ldots, u_m\}$, $u_0 = 0$, and $u_{m+1} = p$ .  For any $i \in \{1, \ldots, m+1\}$, let $p_i = u_i - u_{i-1}$.  For any $i \in \{0, \ldots, m\}$ and $j \in \{1, \ldots, p_{i-1} - 1\}$, by the derivation on Page 12 of \cite{ortelli2019prediction}, it holds that $\|D \psi_{ u_{i-1}+j }^{-S^{\prime}}\|_1 = \sum_{l=1}^{p_{i}} |a_{l+1} - a_l|$, where for $l = 1, \ldots, p_i$, 
\[
  a_l = \frac{l - 1}{p_i} \mathbf{1}\{ l \leq u_{i-1} + j \} - \frac{p_i -l + 1}{p_i} \mathbf{1}\{ l > u_{i-1}+j \}.
\]
It then holds that $\|D \psi_{ u_{i-1}+j }^{-S^{\prime}}\|_1 \leq 2$.  It again follows from the RIC that
\begin{equation}\label{eqn7}
    \max_{j\in \{1, \ldots, p-1\} \backslash S^{\prime}} \frac{\| A\psi_j^{-S^{\prime}}\| }{\| \psi_j^{-S^{\prime}}\| } \leq 1 + \eta_2 + \sqrt{\rho_2(4)},
\end{equation}
since $\|\psi_j^{-S^{\prime}}\| \geq 1/2$ \citep[see Page 12 in][]{ortelli2019prediction}.  

Moreover, it follows from the arguments on Page 12 in \cite{ortelli2019prediction} that, there exists an absolute constant $C_1>0$  such that   
\begin{equation} \label{eqn8}
	\max_{j\in  \{1, \ldots, p-1\} \backslash S^{\prime} } \left\| \psi_j^{-S^{\prime}} \right\| \leq C_1 \sqrt{\frac{p}{s+1}}.
\end{equation}
Combining \eqref{eqn6}, \eqref{eqn7} and \eqref{eqn8}, we have that, on the event $\Omega_1 \cap \Omega_2$, there exists an absolute constant $C_2>0$ such that
\begin{align}
    (II) & \leq C_2 \sigma \sqrt{\log(p \vee n)} (1 + \eta_2 + \sqrt{\rho_2(4)}) \sqrt{\frac{p}{s+1}} \| w_{-S^{\prime}} \{D(\hat{x}-x^*)\}_{-S^{\prime}} \|_1 \nonumber \\
    & \hspace{0.5cm} + \lambda (\|Dx^*\|_1 -  \|D\hat{x}\|_1) + \lambda (V -  V^*). \label{eq-termii-lambda}
\end{align}
Choosing $\lambda = C_2\sigma \sqrt{\log( p \vee n)} (1 + \eta_2 + \sqrt{\rho_2(4)})   \sqrt{p/(s+1)}$, due to \eqref{eq-termii-lambda}, we have that 
\begin{align}\label{eqn10}
    (II) & \leq \lambda (\| w_{-S^{\prime}} \{D(\hat{x}-x^*)\}_{-S^{\prime}} \|_1 + \|Dx^*\|_1 -  \|D\hat{x}\|_1 + V -  V^*) \nonumber \\
    & \leq 2\lambda \Gamma(q_{S^{\prime}},w_{-S^{\prime}}) p^{-1/2} \|\hat{x} -x^*\| + \lambda (V -  V^*)\nonumber \\
    & \leq C_3 \lambda \sqrt{\log\{p/(s+1)\}}(s+1) p^{-1/2} \|\hat{x} - x^*\| + \lambda (V -  V^*) \nonumber \\
    & \leq C_4 \sigma \sqrt{(s+1) \log\{p/(s+1)\} \log (p \vee n)} (1 + \eta + \sqrt{\rho_2(4)}) \|\hat{x} - x^*\| \nonumber \\
    & \hspace{0.5cm} + C_4 \sigma \sqrt{p/(s+1) \log(p \vee n)} (1 + \eta_2 + \sqrt{\rho_2(4)}) (V - V^*),
\end{align}
where $C_3, C_4 > 0$ are absolute constants, $q_{S^{\prime}} = \mathrm{sign}\{(D x^*)_{S^{\prime}}\}$, the second inequality follows from the proof of Theorem 2.2 (Page 25) in \cite{ortelli2019prediction}, the third inequality follows from Section 3.3 in \cite{ortelli2019prediction}, 
 and the last is due to the choice of $\lambda$.

We now have two cases.

\medskip
\noindent \textbf{Case 1.} If $\|\hat{x}-x^*\|\leq 1$, then the final claim follows.

\medskip
\noindent \textbf{Case 2.} If $\|\hat{x}-x^*\|> 1$, then we have that
    \[
        \left \|D\left(\frac{\hat{x}-x^*}{\|\hat{x}-x^*\|}\right)\right\|_1 \leq   \frac{ \|D \hat{x}\|_1  +   \| Dx^*\|_1}{\|\hat{x}-x^*\|} \leq 2V,
    \]
    which, combining with the RIC as in \Cref{def1}, leads to that 
    \begin{equation}\label{eq-a-x-lower-bound}
        \|Ax^* - A\hat{x}\|^2 \geq (1 - \eta_1 -\sqrt{\rho_1(2V)})^2 \|\hat{x} -x^*\|^2.    
    \end{equation}
      
Letting $\gamma = \{1 - \eta_1 -  \sqrt{\rho_1(2V)}\}^{-2}$, combining \eqref{eqn1}, \eqref{eqn3}, \eqref{eqn10} and \eqref{eq-a-x-lower-bound}, we thus have that  
\begin{align} \label{eqn12}
    \|\hat{x} -x^*\|^2 & \leq 2\sqrt{2} \sigma \gamma (1 + \eta_2 + \sqrt{\rho_2(1)})\sqrt{r_{S^{\prime}} \{\log(2 r_{S^{\prime}}) + \log (p \vee n)\}}\|\hat{x} - x^*\| \nonumber \\
    & \hspace{0.5cm} + 2C_4 \sigma \gamma \sqrt{(s + 1) \log\{p/(s+1)\} \log(p \vee n)} (1 + \eta + \sqrt{\rho_2(4)})\|\hat{x} - x^*\| \nonumber \\
    & \hspace{0.5cm} + 2C_4 \sigma \gamma \sqrt{p/(s+1) \log(p \vee n)} (1 + \eta_2 + \sqrt{\rho_2(4)}) (V - V^*).
\end{align}

\smallskip
\noindent \textbf{Case 2.1.} If 
    \[
        \|\hat{x} - x^*\| \leq 2C_4\sigma \gamma \sqrt{p/(s+1) \log(p \vee n)} (1 + \eta_2 + \sqrt{\rho_2(4)}) (V - V^*),
    \]
    then we have that
    \[
        \|\hat{x} -x^*\| ^2 \leq  4C_4^2 \sigma^2 \gamma^2  \log( p \vee n) (1 + \eta_2 + \sqrt{\rho_2(4)})^2 p/(s+1)(V -  V^*)^2
    \]
    and conclude the proof.

\smallskip
\noindent \textbf{Case 2.2.} If 
    \[
        \|\hat{x} - x^*\| > 2C_4\sigma \gamma \sqrt{p/(s+1) \log(p \vee n)} (1 + \eta_2 + \sqrt{\rho_2(4)}) (V - V^*),
    \]
    then \eqref{eqn12} implies that
    \begin{align*}
        \|\hat{x} - x^*\|^2  & \leq C_5\sigma^2 \gamma^2 s \log\{p/(s+1)\} \log(p \vee n) + 2
    \end{align*}
    where $C_5 > 0$ is a constant, and conclude the proof.
\end{proof}

\subsection[]{Proof of \Cref{prop1}} \label{sec:prop2}
\begin{lemma} \label{lem3}
Let  $A \in \R^{p \times p}$,
\[
\mathcal{S}=\{x \in \mathbb{R}^m:~\|x\|= 1, \|Dx\| \leq  t \}
\]
and $\mathcal{N}$ be an $\epsilon$-net of  $\mathcal{S}$ for some  $\epsilon>0$. Then
\[
 \sup_{y\in \mathcal{N}} \left\vert y^{\top} A y\right\vert	  \geq    \sup_{x\in \mathcal{S}} \left\vert x^{\top} A x \right\vert  - 2\|A\|_{\mathrm{op}} \epsilon.
\]
\end{lemma}

\begin{proof}
Let $x _0\in  \mathcal{S}$ such that
\[
	 x_0 \in \argmax_{x \in \mathcal{S}} \left\vert  x^{T}A x \right\vert.
\]
Also, let  $y \in \mathcal{N}$ such that $\|x_0-y\|\leq \epsilon$. Then,
\begin{align}
  \left\vert  x_0^{T}A x_0 - y^{T}A y\right\vert & =    \left\vert       x_0^{\top} A(x_0-y)  +(x_0-y)^{\top} A y \right\vert \nonumber \\
	     & \leq  \|x_0\|  \|A\|_{\mathrm{op} } \|x_0 -y\|   + \|x_0 -y\|\|A\|_{\mathrm{op}} \|y\| \nonumber \\
	      & \leq   \|A\|_{\mathrm{op}} \cdot \epsilon + \epsilon \cdot \|A\|_{\mathrm{op}} \nonumber \\
	       & =  2\epsilon  \|A\|_{\mathrm{op} }. \nonumber
\end{align}
Hence, 
\[
 \sup_{y\in \mathcal{N}} \vert y^{\top} A y \vert	  \geq    \sup_{x\in \mathcal{S}} \left\vert x^{\top} A x \right\vert  - 2\epsilon \|A\|_{\mathrm{op}} ,   
\]
and the lemma follows.
\end{proof}

\begin{proof}[Proof of \Cref{prop1}]

First notice that by Equation. (5.25) in \cite{vershynin2018high},   there exist absolute constants $C_1, c_3>0$ such that  the event
\[
\Omega := \left\{ \left\| A^{\top} A -\Sigma \right\|_{\mathrm{op}} \leq C_1 \sqrt{\frac{p}{n}}+ \sqrt{\frac{\log n}{n}}    \right\}
\]
happens with probability at least $ 1-  2n^{-c_3}$. From now on we assume that the event $\Omega$ holds. Next let
\[
r  := \frac{(1-\zeta)^2}{8\left(C_1 \sqrt{\frac{p}{n}}+ \sqrt{\frac{\log n}{n}}\right) },
\]
and $\mathcal{N}$ be an $r$-net of  $\mathcal{S}$, where $\mathcal{S}$ was defined in Lemma \ref{lem3}. Notice that if
\[
\sup_{x\in \mathcal{S} } \left\vert x^{\top} \left( \Sigma - A^{\top} A \right)x \right\vert \leq \frac{(1-\zeta)^2}{2}
\]
then
\[
1-\zeta - \frac{1 - \zeta}{\sqrt{2}} \leq \left\|Ax \right\|\leq 1+\zeta + \frac{1-\zeta}{\sqrt{2}}
\]
for all $x$. Suppose now that 
\[
 \sup_{x\in \mathcal{S}} \left\vert x^{\top} \left( \Sigma -  A^{\top} A \right)x \right\vert >   \frac{(1-\zeta)^2}{2}. 
\]
Then from Lemma \ref{lem3} it follows that 
\begin{equation}
    \label{eqn:contradiction}
       2\sup_{x\in \mathcal{N}} \left\vert x^{\top} \left( \Sigma - A^{\top} A \right) x \right\vert	 \geq   2\sup_{x\in \mathcal{S}} \left\vert x^{\top} \left( \Sigma - A^{\top} A \right) x \right\vert - 4r \left\| A^{\top} A -\Sigma \right\|_{\mathrm{op}}  > \frac{(1-\zeta)^2}{2}.
\end{equation}
However, as we will see next  the latter happens with small probability.  Towards that end,  let  $y \in \R^p$ with $\|y\|=1$. Let  $z =    \Sigma^{1/2} y$ and $\tilde{A}   =     A \Sigma^{-1/2}$. Then $\| z\|^2 = y^{\top} \Sigma y  \leq  (1 + \zeta)^2$ and so  for any $\varepsilon > 0$ we have that 
\begin{align}
   \mathbb{P}\left(  \left\vert y^{\top} \left( A^{\top}A  - \Sigma \right) y      \right\vert  \geq  \frac{(\zeta+1)^2 \varepsilon}{2}    \right)
	 & \leq    	\mathbb{P}\left(  \left\vert y^{\top} \left(  A^{\top}A  -   \Sigma \right) y      \right\vert  \geq  \frac{ \| z\|^2\varepsilon}{2}    \right)  \nonumber\\
	 & =        	\mathbb{P}\left(  \left\vert  \frac{z^{\top}}{\|z\|} \tilde{A}^{\top}\tilde{A} \frac{z}{\|z\|} -  1     \right\vert  \geq  \frac{ \varepsilon}{2}    \right)  \nonumber\\
	  & \leq  2\exp\left(   -c_4U^{-4} \min\{  \varepsilon,\varepsilon^2  \}n\right),  \nonumber
\end{align}
where $c_4 > 0$ is a sufficiently large absolute constant, and the last inequality follows by the argument on Page 24 of  \cite{vershynin2018high}. Therefore, by union bound,  we have that 
\[
 	\mathbb{P}\left(     \sup_{y\in \mathcal{N} }  \left\vert y^{\top}  \left( A^{\top}A  -  \Sigma \right) y      \right\vert  \geq  \frac{(\zeta+1)^2 \varepsilon}{4}    \right)\leq  2\exp\left(   -c_4U^{-4} \min\{  \varepsilon,\varepsilon^2  \}n  +   \log \vert \mathcal{N}\vert\right).
\]
Furthermore,  by the argument in the proof of Lemma  B.1 in \cite{guntuboyina2020adaptive}, for an absolute constant $C_2 > 0$, we have that 
\[
\log \vert \mathcal{N}\vert  \leq  C_2  \left[t \sqrt{p}  +   \log p\right] \left( \sqrt{\frac{p}{n}}+ \sqrt{\frac{\log n}{n}}\right) \frac{1}{ \left( 1-\zeta \right)^2}.  
\]
Therefore choosing 
\[
\varepsilon = \frac{(1 -\zeta)^2}{4c_2(1 + \zeta)^2} \sqrt{  \left(\frac{t \sqrt{p}  +   \log p}{n}\right)\left( \sqrt{\frac{p}{n}}+ \sqrt{\frac{\log n}{n}}\right) },
\]
by \Cref{eq5}, we obtain that 
\[
 \sup_{y\in \mathcal{N} } \left\vert y^{\top} \left( A^{\top}A  - \Sigma \right) y      \right\vert  \leq \frac{(1 -\zeta)^2}{4},
\]
with probability at least $1-2\exp\left(-c_5 U^{-4} n \epsilon^2\right)$, where $c_5 > 0$ is an absolute constant. Thus, we have arrived at a contradiction to (\ref{eqn:contradiction}) and the claim follows. By \Cref{eq5}, the proposition follows with probability at least $1 - 2n^{-c_3}-2\exp\left(-c_6U^{-4}\phi_n \right)$, where $c_3, c_6 > 0$ are absolute constants.

\end{proof}

\subsection[]{Proof of \Cref{thm_penalise_RIC_bound}}

\begin{lemma}\label{lemma1}
Let $\{ n^{-1/2} a_i\}_{i = 1}^n$ be the rows of matrix $A \in \R^{n \times p}$.  Assume $\{a_i\}_{i = 1}^n$ are independent and identically distributed from $\mathcal{N}(0, \Sigma)$, satisfying that 
    \[
        0 < l \leq \sigma_{\min}(\Sigma)  \quad \mbox{and} \quad \|\Sigma\|_{\infty} \leq \rho,
    \]
    where $l$ and $\rho$ are absolute constants.  Suppose that $\epsilon \sim \mathcal{N}(0, \sigma^2 I_n)$.
	If $\lambda_1 = C_{\lambda_1} \sigma\sqrt{\rho\log (p \vee n)}$, then we have  
\[
\P \left( 2\|A^\top \epsilon \|_\infty \leq  \frac{\lambda_1}{2} \right) \geq 1 - c_1 (p \vee n)^{-1}-p\exp(-n).
\]
where $c_1 > 0$ is an absolute constant.
\end{lemma}

\begin{proof}
Let $A_{,j}$ be the $j^\th$ column of $A$ for $j=1,\dots,p.$ For any constant $C$, we have
\[
	\P \left( \left\vert A_{,j}^\top\epsilon \right\vert \geq C \right) = \P( \left\vert A_{,j}^\top \epsilon\right\vert  \geq C, \left\| A_{,j} \right\| \leq  \sqrt{5\Sigma_{j, j}})+\P \left(  |A_{,j}^\top \epsilon| \, \geq C , \left\| A_{,j} \right\|  >  \sqrt{5\Sigma_{j, j}} \right).
\]
Noticing that each element in $A_{,j}$ is a realization, we have 
\[
 \P \left(  \left\vert A_{,j}^\top \epsilon \right\vert \geq  C, \left\| A_{,j} \right\| \leq  \sqrt{5\Sigma_{j,j}}\right)
 \leq 2 \exp \left( \frac{-C^2}{10\sigma^2\Sigma_{j, j}} \right),
\]
where we get the last step by sub-Gaussian tails, and $A$ and the errors $\epsilon$ are independent.
\begin{align}
    \P \left( \left\vert A_{,j}^\top \epsilon \right\vert \geq C,  \left\|A_{,j} \right\| \right)  >   \sqrt{ 5\Sigma_{j,j} }) 
    & \leq    \P \left( \left\| A_{,j} \right\|  > \sqrt{5\Sigma_{j, j}} \right) \nonumber\\
 	& =  \P \left(\frac{ \left\| A_{,j} \right\|^2}{\Sigma_{j, j}} > 5 \right) \nonumber\\
 	& \leq  \exp(-n), \nonumber
\end{align}
 where we get the last step by Lemma 1 in \cite{laurent2000adaptive}. Then we have 
\[
	\P \left( \left\vert A_{,j}^\top \epsilon \right\vert  \geq C \right)  \leq  2 \exp \left( \frac{-C^2}{10\sigma^2\Sigma_{j, j}} \right) +  \exp \left( -n \right).
\]
Using the union bound, we have  	
\[
	\P \left( 2\left\| A^\top\epsilon \right\|_{\infty}  \geq  C \right) \leq 2 \exp \left(\frac{-C^2}{40\sigma^2\rho}+\log p \right) + \exp\left(-n+\log p\right)
\]
Let $C = \lambda_1/2$, we have 
\[
 \P( 2 \| A^\top \epsilon \|_{\infty} \, \leq \, \frac{\lambda_1}{2})\geq 1 - c_1  (p \vee n)^{-1} - p\exp(-n),
\]
where $c_1 > 0$ is an absolute constant, completing the proof.	
\end{proof}

\begin{proof}[Proof of \Cref{thm_penalise_RIC_bound}]
First, by the basic inequality
\[
	 \left\| y-A\tilde x \right\|^2 + \lambda_1 \left\| \tilde  x \right\|_1 + \lambda_2 \left\| D\tilde  x \right\|_1  \leq  
	 \left\| y-A x^* \right\|^2 + \lambda_1 \left\| x^* \right\|_1 + \lambda_2 \left\| Dx^* \right\|_1.
\]
By rearranging, we have 
\[
	  \left\| A(\tilde x-x^*) \right\|^2  \leq  
	 2 \epsilon^\top A(\tilde x-x^*) + \lambda_1 \left(\left\| x^* \right\|_1 - \left\| \tilde x \right\|_1 \right)+\lambda_2 \left( \left\|D x^* \right\|_1 - \left\| D \tilde x \right\|_1 \right).
\]
Now $\left\vert \epsilon ^ \top A \left( \tilde x- x^* \right) \right\vert  \leq  \left\| A^\top \epsilon  \right\|_\infty  \left\| \tilde x - x^* \right\|_1$. Let $\Omega_1 : = \left\{ 2 \left\|A^ \top \epsilon \right\|_\infty \leq \frac{\lambda_1}{2} \right\}$. Lemma \ref{lemma1} shows that $\P \left( \Omega_1 \right)  \geq   1 - c_1 \left(p \vee n \right)^{-1} - p\exp(-n),$ where $c_1 > 0$ is an absolute constant. Working on the event $\Omega_1,$ we obtain 
\begin{equation}\label{equation1}
   \left\| A \left(\tilde x - x^* \right) \right\|^2   \leq 
	 \lambda_1/2  \left\| \tilde x - x^* \right\|_1
	 +\lambda_1\left( \left\| x^* \right\|_1 - \left\|\tilde x \right\|_1 \right) + \lambda_2 \left( \left\| Dx^* \right\|_1 - \left\|D \tilde x \right\|_1 \right).  
\end{equation}
Then for the first two terms in the right hand side of the above inequality, we have
\begin{align}\label{equation2} 
	  \lambda_1/2  \left \|\tilde x - x^* \right\|_1 + \lambda_1 \left(  \left\| x^* \right\|_1 - \left\|\tilde x \right\|_1  \right)	 
	& =  \lambda_1/ 2    \left( \left\| \tilde x_H - x^*_H  \right\|_1 + \left\| \tilde x_{-H} \right\|_1 \right)
	 + \lambda_1   \left( \left\| x^*_H \right\|_1 -  \left\| \tilde x_H \right\|_1 -  \left\|\tilde x_{-H} \right\|_1  \right) \nonumber\\
    & \leq   3\lambda_1/2  \left\| \tilde x_H - x^*_H \right\|_1 \nonumber\\
	& \leq   3\lambda_1/2\, \sqrt{h}  \left\| \tilde x - x^*  \right\|, 
\end{align}
where we get the first inequality by the triangle inequality and the second inequality by Cauchy-Schwarz inequality. 
Similarly,
\begin{align}\label{equation3}
	 \left\| Dx^* \right\|_1  -  \left\| D\tilde x \right\|_1 
	  & =   \left\| D_S x^* \right\|_1 -  \left\| D_S \tilde x \right\|_1  -  \left\| D_{-S}\tilde x \right\|_1 \nonumber\\
	  & \leq  \left\| D_S  \left( x^* -  \tilde x \right) \right\|_1 \nonumber\\ 
	  & \leq  2 \sqrt{s+1}  \left\| \tilde x  -  x^*  \right\|.
\end{align}
Let $\Omega_2 : =  \left\{  \left\| A \left(\tilde x  -   x^* \right) \right\|^2  \geq \frac{l}{32}  \left\| \tilde x  -  x^* \right\|^2-81\rho^2 \frac{\log p}{n} \left\| \tilde x -  x^* \right\|_1^2  \right\}$. Lemma 9 in \cite{zhang2015change} shows  that $\P \left( \Omega_2 \right) \geq  1-c_2\exp(-c_3n),$ where $c_2, c_3 > 0$ are absolute constants. Working on the event $\Omega_1\cap\Omega_2$ and combining (\ref{equation1}), (\ref{equation2}), and (\ref{equation3}) we have 
\[
		\frac{l}{32} \left\| \tilde x -  x^* \right\|^2 \leq  81\rho^2  \frac{\log p}{n} \left\|\tilde x - x^* \right\|_1^2 + 3\lambda_1/2 \sqrt{h} \left\| \tilde x - x^* \right\|  +  2\lambda_2  \sqrt{s+1} \left\|\tilde x  - x^* \right\|.
\]
Multiplying both sides by $\frac{64}{l},$ and using the simple inequality $ab \, \leq \,  a^2/2 \, + \, b^2/2$ twice, we get 
\[
    2 \left\| \tilde x - x^* \right\|^2 \leq  \frac{64\times81\rho^2}{l}  \frac{\log p}{n}  \left\|\tilde x - x^* \right\|_1^2  +  \frac{96^2\lambda_1^2}{2l^2}h  + \frac{ \left\| \tilde x  - x^* \right\|^2 }{2}  + \frac{128^2\lambda_2^2}{2l^2}  \left( s + 1 \right) + \frac{\left\| \tilde x  - x^* \right\|^2 }{2}.
\]
Then we have 
\begin{equation}\label{equation6}
   \left\| \tilde x  - x^*  \right\|^2  \leq  \frac{64\times81\rho^2}{l}  \frac{\log(p)}{n} \left\|\tilde x - x^* \right\|_1^2 + \frac{96^2\lambda_1^2}{2l^2}h  + \frac{128^2\lambda_2^2}{2l^2} \left( s+1 \right).
 \end{equation}
Next, we need to give an upper bound of $\|\tilde x \,- \, x^*\|_1^2.$ By rearranging (\ref{equation1}), we have 
\[
	 \lambda_1 \left\| \tilde x \right\|_1  \leq  
	   \left\|\tilde x - x^* \right\|_1 /2
	 + \lambda_1 \left\| x^* \right\| + \lambda_2 \left( \left\| Dx^* \right\|_1 - \left\| D\tilde x \right\|_1 \right). 
\]
Then by the triangle inequality, we have
\[
	 	\left\| \tilde x_{-H} \right\|_1 /2 \leq  
	 3\lambda_1  \left\|\tilde x_H - x^*_H \right\|_1/2
	 + \lambda_2 \left( \left\| Dx^* \right\|_1  -  \left\| D\tilde x \right\|_1 \right).
\]
Add $ \left\|\tilde x_H - x^*_H \right\|_1$ to both sides, and since $\lambda_2 / \lambda_1 = \frac{C_{\lambda_2}}{ C_{\lambda_1} \sqrt{\rho} }$ we have
\begin{equation}\label{equation4}
	 	  \left \| \tilde x  -  x^* \right\|_1  \leq  
	4  \left\| \tilde x_H - x^*_H \right\|_1 +  \frac{C_{\lambda_2}}{  C_{\lambda_1} \sqrt{\rho} } \left( \left\| Dx^* \right\|_1  -  \left\| D\tilde x \right\|_1 \right).
\end{equation}
Substitute inequalities (\ref{equation2}) and (\ref{equation3}) into (\ref{equation4}), we have
\[
		 \left\| \tilde x  -  x^* \right\|_1 \leq 
	  4\sqrt{h} \left\|\tilde x  -  x^* \right\| 
	 +  \frac{ 2C_{\lambda_2}}{ C_{\lambda_1} \sqrt{\rho} }  \sqrt{s+1}  \left\|\tilde x  -  x^* \right\|.
\]
Using the simple inequality $(a+b)^2\leq2a^2+2b^2,$ we have 
\begin{equation}\label{equation5}
		 \left\| \tilde x  -  x^* \right\|_1^2 \leq 
	\left[ 32h  +  \frac{8C_{\lambda_2}^2}{ C_{\lambda_1} ^2\rho }(s+1) \right]   \left\| \tilde x  -  x^* \right\|^2.
\end{equation} 
Substitute (\ref{equation5}) into (\ref{equation6}), we have 
\[
 \left\| \tilde x  - x^*  \right\|^2  \leq  \frac{64\times81\rho^2}{l}  \frac{\log(p)}{n} 	\left[ 32h  +  \frac{8C_{\lambda_2}^2}{ C_{\lambda_1} ^2\rho }(s+1) \right]   \left\| \tilde x  -  x^* \right\|^2 + \frac{96^2\lambda_1^2}{2l^2}h  + \frac{128^2\lambda_2^2}{2l^2} \left( s+1 \right).
\]
Since we assume $(s+h) \log p \leq  C_1 n$, where $C_1$ is a sufficiently large constant, then there exists a absolute constant  $0  <  C_2  < 1$ depending on $\rho$ such that
\[
   \frac{64\times81\rho^2}{l}  \frac{\log(p)}{n} 	\left[ 32h  +  \frac{8C_{\lambda_2}^2}{ C_{\lambda_1} ^2\rho }(s+1) \right] \leq C_2.
\]
Thus, we have 
\[
  \left \| \tilde x  -  x^* \right\|^2  \leq   \frac{1}{1-C_2}  \left[  \frac{96^2\lambda_1^2}{2l^2}  h  +\ \frac{128^2\lambda_2^2}{2l^2} (s+1) \right],
\]
which is equivalent to 
\[
	   \left \| \tilde x  -  x^* \right\|^2  \leq   C \sigma^2 \log( p \vee n) (s \vee h),
\]
where $C$ is a absolute constant depending on $l$ and $\rho$, completing the proof.
\end{proof}

\section[]{PROOFS OF THE RESULTS IN SECTION \ref{sec-post-processing}}

\begin{proof}[Proof of \Cref{thm_number_change_points}]
Let $\Omega = \left\{ d_\text{H}\left( S_I\left(\hat x \right), S\right) \leq  b_{p, n}\right\}$. By \Cref{thm_two_side_hausdorff_distance}, we have $\P (\Omega) \geq 1 - 2 (p \vee n)^{-1}$. Working on the event $\Omega$, by the definition of Hausdorff distance, we have that for any true change point, there must exist at least one  estimated change point in $S_I\left( \hat x\right)$ such that the distance between them less than or equal to $b_{p, n}$, i.e, the set 
\[
     E_{t_{i}} = \left\{ \hat t_{j}:~\hat t_{j} \in S_I(\hat x) \, \,\mbox{and}\, \, \left\vert \hat t_{j} -t_{i} \right\vert \leq b_{p, n} \right\},
\]
for any $t_i \in S$ is not empty. We also have for any estimated change points there also exist at least one  true change point such that the distance between them less that or equal to $b_{p, n}$, i.e, $S_I(\hat x) = \underset{t_i \in S} {\cup}  E_{t_{i}}$.
 We have $S_T(\hat x)$ is a subset of $S_I(\hat x)$ by construction. Thus, we have  $S_T(\hat x) = \underset{t_i \in S} {\cup}  T_{t_{i}}$, where
\[
     T_{t_{i}} = \left\{ \hat t_{j}:~\hat t_{j} \in S_T(\hat x) \, \,\mbox{and}\, \, \left\vert \hat t_{j} -t_{i} \right\vert \leq b_{p, n} \right\},
\]
for any $t_i \in S$. Also, by construction of $S_T(\hat x)$ and $t_{p, n} = 2b_{p, n}$, we have $ \left \vert T_{t_{i}} \right\vert \leq 1$, for any $t_i \in S$. Since we assume $\Delta_{p, n} > 4b_{p, n}$, we have  $\left \vert T_{t_{i}} \right\vert = 1$, and sets  $T_{t_{i}}$ for $t_i \in S$ are disjoint. Finally, we have  $ \left \vert S_T(\hat x) \right\vert = \left\vert S \right\vert $, completing the proof.
\end{proof}

\section[]{ADDITIONAL DETAILS AND RESULTS IN SECTION \ref{sec-numerical}}\label{simulation_appendix}

\subsection{Details of the Three Scenarios Considered for Each of the Four Types of Design Matrices}

Recall that the three scenarios are only one change point, nine equally-space change points and nine unequally-spaced change points.  The details are presented below.
\begin{itemize}
    \item Only one change point: $\Delta_{p, n} = 500$ and $\kappa_{p, n} = \gamma$. 
    \item Nine equally-spaced change points: $\Delta_{p, n} = 100$ and $\{ x^*_{t_1}, \ldots, x^*_{t_{10}}\} = \{0, \gamma,  0,   1.5\gamma,  0,  2\gamma,  0,  1.75\gamma,  0,  0.75\gamma\}$.   
    \item Nine unequally-spaced change points: $\{ t_1, \dots, t_9\} = \{200,  310,  360,  390,   450,  490, 570,  640,   770,   970\}$ with $\Delta_{p, n} = 30$ and $\{x^*_{t_1}, \dots, x^*_{t_{10}}\} = \{0, \gamma,  0,   1.5\gamma,  0,  2\gamma,  0,  1.75\gamma,  0,  0.75\gamma\}$.
\end{itemize}

\subsection{The Identity Design Matrix Case}
Let $n = p = 1000$ and design matrix  $A = I_n$. \Cref{table_identity_matrix_one} displays the results of the first scenario: only one change point for $\gamma = 1$ and noise levels $\sigma \in \left\{1, 2, 4 \right\}$.  \Cref{table_identity_matrix_nine_equal} displays the results of the second scenario of the identity matrix:  $\left| S \right| = 9$ and changes points are equally spaced for $\gamma = 1$ and noise levels $\sigma \in \{1, 2, 4\}$. \Cref{table_identity_matrix_nine_unequal} displays the results of the third scenario of the identity matrix : $\left| S \right| = 9$ and changes points are unequally spaced for $\gamma = 1$ and noise levels $\sigma \in \{0.5, 1, 2\}$.

\begin{table}[h]
\caption{Change point analysis results: Identity matrix and only one change point, for $\gamma = 1$ and noise levels $\sigma = \{1, 2, 4\}$. The reported values are in the form of mean (standard deviation) across $100$ independent simulations.  The bold-faced entries indicate the best method in different measures, with the convention that if $S(\hat x) = \emptyset$, then $d(S_0|S(\hat x)) = d(S_0|S(\hat x)) = p$.}\label{table_identity_matrix_one}
\begin{center}
\begin{tabular}{llllll}
   & & \textbf{FL} & \textbf{FLMF} & \textbf{FLMTF} & \textbf{ITALE}\\ \hline
 
 \multirow{3}{*}{ $\sigma = 1 $} & $d(S(\hat x)|S_0) $ &$\bf{1.13 (1.62)} $& $ 42.61 (196.74)$ & $ 44.08 (196.47)$ & $ 2.85 (5.53)$\\
 & $d(S_0|S(\hat x))$ & $268.43 (148.60) $& $ 71.71 (213.47)$ & $68.64 (214.36)$ & $ \bf{34.17 (113.35)}$ \\
 & $|S(\hat x) - S|$ & $7.70 (6.87)$ & $2.24 (2.14)$&$ \bf{0.15 (0.44)}$& $0.20 (0.77)$\\ \hline

\multirow{3}{*}{ $\sigma = 2$} & $d(S(\hat x)|S_0)$ &$\bf{2.80 (3.78)}$ & $154.82 (356.98)$ & $155.65 (356.63)$ & $6.40 (8.76)$\\
 & $d(S_0|S(\hat x))$ & $247.95 (146.89)$ & $165.58 (354.58)$ & $163.81 (355.35)$ & $ \bf{27.53 (77.67)}$ \\
 & $|S(\hat x) - S|$ & $ 7.05 (6.34)$ & $1.36 (1.36)$ & $0.26 (0.50)$ & $\bf{0.25 (0.86)}$\\ \hline

\multirow{3}{*}{ $\sigma = 4$} & $d(S(\hat x)|S_0)$ & $\bf{5.58 (6.01)}$ & $408.93 (479.58)$ & $409.26 (479.30)$ & $14.31 (18.97)$\\
 & $d(S_0|S(\hat x))$ & $292.22 (151.84)$ & $428.27 (470.39)$ & $ 428.61 (470.63)$ & $ \bf{42.67 (99.16)}$ \\
 & $|S(\hat x) - S|$ & $7.60 (7.23)$ & $0.78 (0.73)$ & $0.47 (0.50)$ & $\bf{0.24 (0.75)}$\\ 

\end{tabular}
\end{center}
\end{table}

\begin{table}[h]
\caption{Change point analysis results: Identity matrix and nine change points with equal spacing, for $\gamma = 1$ and noise levels $\sigma = \{1, 2, 4\}$. The reported values are in the form of mean (standard deviation) across $100$ independent simulations. The bold-faced entries indicate the best method in different measures, with the convention that if $S(\hat x) = \emptyset$, then $d(S_0|S(\hat x)) = d(S_0|S(\hat x)) = p$.}\label{table_identity_matrix_nine_equal}
\begin{center}
\begin{tabular}{llllll}
   & & \textbf{FL} & \textbf{FLMF} & \textbf{FLMTF} & \textbf{ITALE}\\ \hline
 
 \multirow{3}{*}{ $\sigma = 1 $} & $d(S(\hat x)|S_0) $ &$\bf{2.98 (2.60)} $& $90.55 (54.25)$ & $ 92.96 (54.64)$ & $ 17.05 (26.47)$\\
 & $d(S_0|S(\hat x))$ & $ 75.17 (13.91) $& $ 12.50 (10.81)$ & $\bf{ 8.62 (11.93)}$ & $ 12.36 (14.68)$ \\
 & $|S(\hat x) - S|$ & $39.24 (9.25)$ & $17.21 (6.52)$&$ 1.34 (0.95)$& $\bf{ 0.34 (0.68)}$\\ \hline

\multirow{3}{*}{ $\sigma = 2$} & $d(S(\hat x)|S_0)$ &$\bf{6.74 (5.88)}$ & $162.57 (93.95)$ & $ 164.22 (93.90)$ & $69.39 (67.41)$\\
 & $d(S_0|S(\hat x))$ & $ 71.98 (13.35)$ & $ 14.17 (13.84)$ & $\bf{ 11.07 (14.29)}$ & $ 20.16 (16.96)$ \\
 & $|S(\hat x) - S|$ & $35.48 (8.36)$ & $ 9.20 (6.59)$ & $ 2.62 (1.41)$ & $\bf{ 1.32 (2.25)}$\\ \hline

\multirow{3}{*}{ $\sigma = 4$} & $d(S(\hat x)|S_0)$ & $\bf{12.14 (8.43)}$ & $ 256.40 (121.36)$ & $257.18 (121.28)$ & $ 150.82 (96.14)$\\
 & $d(S_0|S(\hat x))$ & $71.17 (16.04)$ & $14.83 (15.76)$ & $ \bf{13.59 (16.08)}$ & $31.68 (22.11)$ \\
 & $|S(\hat x) - S|$ & $29.19 (7.84)$ & $3.55 (2.86)$ & $4.41 (1.60)$ & $\bf{2.24 (1.86)}$\\

\end{tabular}
\end{center}
\end{table}

\begin{table}[h]
\caption{Change point analysis results: Identity matrix and nine change points with unequal spacing, for $\gamma = 1$ and noise levels $\sigma = \{0.5, 1, 2\}$. The reported values are in the form of mean (standard deviation) across $100$ independent simulations. The bold-faced entries indicate the best method in different measures, with the convention that if $S(\hat x) = \emptyset$, then $d(S_0|S(\hat x)) = d(S_0|S(\hat x)) = p$.}
\label{table_identity_matrix_nine_unequal}
\begin{center}
\begin{tabular}{llllll}
   & & \textbf{FL} & \textbf{FLMF} & \textbf{FLMTF} & \textbf{ITALE}\\ \hline

\multirow{3}{*}{$\sigma = 0.5$} & $d(S(\hat x)|S_0)$ &$\bf{1.49 (1.33)}$ & $33.97 (69.92)$ & $ 43.12 (68.58)$ & $ 10.98 (34.03)$\\
 & $d(S_0|S(\hat x))$ & $275.59 (26.11)$ & $31.85 (59.19)$ & $27.68 (59.94)$ & $\bf{26.41 (63.86)}$ \\
 & $|S(\hat x) - S|$ & $46.87 (9.52)$ & $ 29.17 (7.65)$ & $ 0.65 (0.70)$ & $\bf{0.62 (1.38)}$\\ \hline

\multirow{3}{*}{$\sigma = 1$} & $d(S(\hat x)|S_0)$ &$\bf{ 8.83 (8.96)}$ & $178.60 (74.31)$ & $181.26 (73.92)$ & $124.98 (82.86)$\\
 & $d(S_0|S(\hat x))$ & $275.53 (24.91)$ & $28.36 (58.49)$ & $\bf{25.07 (58.78)}$ & $25.28 (51.14)$ \\
 & $|S(\hat x) - S|$ & $38.28 (13.73)$ & $9.55 (8.52)$ & $2.74 (1.38)$ & $\bf{1.80 (1.29)}$\\ \hline

\multirow{3}{*}{$\sigma = 2$} & $d(S(\hat x)|S_0)$ &$\bf{3.65 (3.22)}$ & $124.54 (86.36)$ & $128.70 (86.66)$ & $62.47 (82.11)$\\
 & $d(S_0|S(\hat x))$ & $278.66 (25.02)$ & $21.59 (40.03)$ & $\bf{17.72 (41.07)}$ & $27.45 (58.93)$ \\
 & $|S(\hat x) - S|$ & $42.08 (8.66)$ & $17.43 (6.94)$ & $1.40 (1.02)$ & $\bf{0.91 (1.30)}$\\
 
\end{tabular}
\end{center}
\end{table}

\subsection{The Band Design Matrix Case} 

\Cref{table_band_matrix_one} displays the results of the first scenario of band matrix: only one change point, with bandwidth $h \in \{ 1, 5, 10, 50\}$, noise level $\sigma \in \{1, 2, 4, 8\}$ and $\gamma \in \{0.25, 0.5, 1, 2\}$. \Cref{table_band_matrix_nine_equal} displays the results of the second scenario of band matrix:  $\left| S \right| = 9$  and change points are equally spaced, with bandwidth $h \in \{ 1, 5, 10, 50\}$, noise level $\sigma \in \{1, 2, 4, 8\}$ and $\gamma \in \{0.25, 0.5, 1, 2\}$. \Cref{table_band_matrix_nine_unequal} displays the results of the third scenario of band matrix:  $\left| S \right| = 9$  and change points are unequally spaced, with bandwidth $h \in \{ 1, 5, 10, 50\}$, noise level $\sigma \in \{0.5, 1, 2, 4\}$ and $\gamma \in \{0.25, 0.5, 1, 2\}$. 

\begin{table}[h]
\caption{Change point analysis results: band matrix and only one change point, with bandwidth $h \in \{ 1, 5, 10, 50\}$, noise level $\sigma \in \{1, 2, 4, 8\}$ and $\gamma \in \{0.25, 0.5, 1, 2\}$. The reported values are in the form of mean (standard deviation) across $100$ independent simulations. The bold-faced entries indicate the best method in different measures, with the convention that if $S(\hat x) = \emptyset$, then $d(S_0|S(\hat x)) = d(S_0|S(\hat x)) = p$.}\label{table_band_matrix_one}
\begin{center}
\begin{tabular}{llllll}
   & & \textbf{FL} & \textbf{FLMF} & \textbf{FLMTF} & \textbf{ITALE}
\\ \hline
\multirow{3}{*}{ $\bf{h = 1}$, $\sigma = 2, \gamma = 1$} & $d(S(\hat x)|S_0)$ &  $\bf{0.82 (1.14)}$& $1.10 (3.42)$ & $2.76 (4.01)$ & $4.91 (7.02)$ \\
 & $d(S_0|S(\hat x))$ & $293.78 (148.88)$ & $20.70 (55.78)$ & $\bf{16.66 (56.57)}$ & $148.02 (194.35)$ \\
 & $|S(\hat x) - S|$ & $7.67 (6.71)$  & $2.89 (2.42)$ &$\bf{ 0.15 (0.41)}$  & $2.15 (3.30)$\\ \hline

\multirow{3}{*}{ $\bf{h = 5}$, $\sigma = 2, \gamma = 1$} & $d(S(\hat x)|S_0)$ & $0.13 (0.44)$ & $\bf{0.11 (0.40)}$ & $2.02 (2.07)$ & $2.07 (3.02)$\\
 & $d(S_0|S(\hat x))$ & $277.38 (41.19)$ & $19.25 (52.05)$ & $\bf{12.85 (53.12)}$ & $41.04 (105.91)$ \\
 & $|S(\hat x) - S|$ &  $8.19 (5.09)$ & $4.97 (3.38)$ & $\bf{0.07 (0.26)}$ & $0.58 (1.37)$\\ \hline

\multirow{3}{*}{ $\bf{h = 10}$, $\sigma = 2, \gamma = 1$} & $d(S(\hat x)|S_0)$ &$\bf{0.08 (0.31)}$ & $\bf{0.08 (0.31)}$ & $1.71 (1.55)$ &$0.95 (1.47)$\\
 & $d(S_0|S(\hat x))$ & $276.97 (137.46)$ & $22.85 (61.57)$ & $15.88 (62.97)$ & $\bf{11.86 (52.02)}$ \\
 & $|S(\hat x) - S|$ & $8.70 (6.10)$ & $4.93 (2.83)$ & $\bf{0.09 (0.32)}$ & $0.40 (1.27)$\\ \hline

\multirow{3}{*}{ $\bf{h = 50}$, $\sigma = 2, \gamma = 1 $} & $d(S(\hat x)|S_0)$ &$\bf{0.00 (0.00)}$ & $\bf{0.00 (0.00)}$ & $1.71 (1.77)$ & $0.01 (0.10)$\\
 & $d(S_0|S(\hat x))$ & $253.69 (152.57)$ & $32.79 (87.30)$ & $24.94 (89.13)$ & $\bf{20.75 (76.29)}$ \\
 & $|S(\hat x) - S|$ & $8.80 (7.82)$ & $5.51 (3.29)$ & $\bf{ 0.11 (0.47)}$ & $0.19 (0.63)$\\ \hline

\multirow{3}{*}{ $h = 10$, $\bf{\sigma = 1}$, $\gamma = 1$} & $d(S(\hat x)|S_0)$ &  $\bf{0.01 (0.10)}$ & $\bf{0.01 (0.10)}$& $2.00 (1.95)$ & $0.31 (0.63)$ \\
 & $d(S_0|S(\hat x))$ & $263.99 (162.93)$ & $43.15 (111.040$ &  $\bf{35.88 (113.09)}$ & $40.53 (101.97)$ \\
 & $|S(\hat x)-S|$ & $9.35 (6.82)$ & $5.36 (3.15)$ &$\bf{0.12 (0.38)}$  & $0.77 (1.48)$\\ \hline

\multirow{3}{*}{$h = 10$, $\bf{\sigma = 4}$, $\gamma = 1$} & $d(S(\hat x)|S_0)$ & $0.17 (0.40)$ & $\bf{0.14 (0.35)}$ & $2.23 (2.21)$  & $1.34 (2.28)$\\
 & $d(S_0|S(\hat x))$ & $271.85 (147.15)$ & $29.90 (89.45)$ & $\bf{23.88 (90.75)}$ & $27.48 (91.32)$ \\
 & $|S(\hat x)-S|$ & $8.35 (6.00)$ & $4.37 (3.26)$ & $\bf{0.12 (0.46)}$& $0.29 (0.76)$\\ \hline

\multirow{3}{*}{$h = 10$, $\bf{\sigma = 8}$, $\gamma = 1$} & $d(S(\hat x)|S_0)$ & $0.42 (0.94)$ & $\bf{0.38 (1.19)}$ & $1.89 (2.36)$ & $2.47 (4.99)$\\
 & $d(S_0|S(\hat x))$ & $257.18 (146.17)$ & $22.67 (71.03)$ & $\bf{17.41 (71.92)}$ & $22.22 (83.53)$ \\
 & $|S(\hat x)-S|$ & $7.62 (6.41)$ & $3.32 (2.32)$ & $\bf{0.13 (0.42)}$ & $0.22 (0.96)$\\ \hline

\multirow{3}{*}{ $h = 10, \sigma = 2$, $\bf{\gamma= 0.25}$} & $d(S(\hat x)|S_0)$ &  $\bf{1.89 (2.49)}$ & $3.70 (10.76) $& $5.00 (10.67) $& $5.96 (8.49)$ \\
 & $d(S_0|S(\hat x))$ & $237.72 (157.50)$ & $21.56 (65.00)$ & $\bf{19.27 (65.41)}$ & $38.32 (104.74)$ \\
 & $|S(\hat x) - S|$ & $7.17 (7.41)$ & $1.35 (1.22)$ & $\bf{0.13 (0.34)}$  & $0.23 (0.72)$\\ \hline

 \multirow{3}{*}{ $h = 10,  \sigma = 2$, $\bf{\gamma = 0.5}$} & $d(S(\hat x)|S_0)$ & $0.55 (0.95)$ & $\bf{ 0.41 (0.75)}$ & $2.04 (2.06)$ & $2.57 (4.18)$\\
 & $d(S_0|S(\hat x))$ & $284.98 (151.33)$ & $16.73 (46.16)$ & $\bf{12.34 (46.93)}$ & $42.24 (120.00)$ \\
 & $|S(\hat x) - S |$ & $8.63 (6.43)$ & $3.68 (2.92)$ & $\bf{0.11 (0.31)}$ & $0.45 (1.40)$\\ \hline

\multirow{3}{*}{ $h = 10, \sigma = 2$, $\bf{\gamma = 2}$} & $d(S(\hat x)|S_0)$ &$\bf{0.00 (0.00)}$ & $\bf{0.00 (0.00)}$& $1.68(1.80)$ & $0.25 (0.76)$\\
 & $d(S_0|S(\hat x))$ & $288.39 (147.41)$ & $22.01 (71.90)$ & $\bf{14.40 (73.14)}$ & $16.76 (73.52)$ \\
 & $|S(\hat x) - S |$ & $ 9.32 (6.76)$ & $5.01 (2.85)$ & $\bf{0.04 (0.24)}$ & $0.37 (0.97)$\\

\end{tabular}
\end{center}
\end{table}

\begin{table}[h]
\caption{Change point analysis results: band matrix and nine equally spaced change points, with bandwidth $h \in \{ 1, 5, 10, 50\}$, noise level $\sigma \in \{1, 2, 4, 8\}$ and $\gamma \in \{0.25, 0.5, 1, 2\}$. The reported values are in the form of mean (standard deviation) across $100$ independent simulations. The bold-faced entries indicate the best method in different measures, with the convention that if $S(\hat x) = \emptyset$, then $d(S_0|S(\hat x)) = d(S_0|S(\hat x)) = p$.}
\label{table_band_matrix_nine_equal}
\begin{center}
\begin{tabular}{llllll}
   & & \textbf{FL} & \textbf{FLMF} & \textbf{FLMTF} & \textbf{ITALE}
\\ \hline
\multirow{3}{*}{ $\bf{h = 1}$, $\sigma = 2, \gamma = 1$} & $d(S(\hat x)|S_0)$ &  $\bf{2.71 (1.86)}$ & $53.19 (56.22) $& $ 56.09 (56.68)$ & $9.80 (8.76)$ \\
 & $d(S_0|S(\hat x))$ & $73.91 (15.48)$ & $16.15 (14.92)$ & $\bf{11.39 (16.49)}$ & $62.04 (23.93)$ \\
 & $|S(\hat x) - S|$ & $41.26 (8.83)$ & $23.58 (8.37)$ & $\bf{0.75 (0.86)}$  & $20.96 (17.94)$\\ \hline

\multirow{3}{*}{ $\bf{h = 5}$, $\sigma = 2, \gamma = 1$} & $d(S(\hat x)|S_0)$ & $\bf{0.61 (0.79)}$ & $1.52 (9.18)$ & $4.65 (9.74)$ & $3.38 (2.32)$ \\
 & $d(S_0|S(\hat x))$ & $73.63  (14.59)$ & $14.04 (10.25)$ & $\bf{6.00 (11.02)}$ & $34.83 (29.15)$ \\
 & $|S(\hat x) - S|$ & $43.55 (9.09)$ & $34.59 (7.87)$ & $\bf{0.06 (0.24)}$& $6.74 (6.18)$\\ \hline

\multirow{3}{*}{ $\bf{h = 10}$, $\sigma = 2,  \gamma = 1$} & $d(S(\hat x)|S_0)$ & $0.27 (0.60)$ & $\bf{0.21 (0.43)}$ & $3.30 (1.34)$ & $1.95 (1.54)$\\
 & $d(S_0|S(\hat x))$ & $71.06 (14.26)$  & $13.44 (7.96)$ & $\bf{4.92 (8.37)}$ & $20.36 (20.83)$ \\
 & $|S(\hat x)|$ & $44.21 (7.39)$ & $40.55 (7.57)$ & $\bf{0.04 (0.20)}$& $3.46 (2.97)$ \\ \hline

\multirow{3}{*}{ $\bf{h = 50}$, $\sigma = 2,  \gamma = 1$} & $d(S(\hat x)|S_0)$ &$\bf{0.00 (0.00)}$ & $\bf{0.00 (0.00)}$ & $3.82 (1.74)$ & $0.09 (0.32)$\\
 & $d(S_0|S(\hat x))$ & $67.10 (15.22)$ & $15.44 (12.46)$ & $7.73 (14.10)$ & $\bf{1.56 (6.92)}$ \\
 & $|S(\hat x) - S|$ & $44.93 (8.72)$ & $45.60 (8.70)$ & $\bf{ 0.11 (0.40)}$& $0.17 (0.49)$\\ \hline

\multirow{3}{*}{ $h = 10$, $\bf{\sigma = 1}$, $\gamma = 1$} & $d(S(\hat x)|S_0)$ &  $\bf{0.07 (0.26)}$ &$\bf{0.07 (0.26)}$ & $3.80 (1.72)$ & $1.22 (1.10)$ \\
 & $d(S_0|S(\hat x))$ & $71.53 (14.67)$ & $14.44 (8.33)$ & $\bf{6.88 (10.13)}$ & $21.53 (23.27)$ \\
 & $|S(\hat x) - S|$ & $45.24 (9.92)$ & $44.75 (8.78)$ & $\bf{0.12 (0.38)}$  & $4.22 (3.45)$\\ \hline

\multirow{3}{*}{ $h = 10$, $\bf{\sigma = 4}$, $\gamma = 1$} & $d(S(\hat x)|S_0)$ & $0.60 (0.77)$ &$\bf{0.49 (0.70)}$ & $3.79 (1.46)$ & $3.48 (2.77)$\\
 & $d(S_0|S(\hat x))$ & $70.73 (15.05)$ & $14.80 (9.74)$ & $\bf{6.76 (10.97)}$ & $20.25 (21.21)$ \\
 & $|S(\hat x) - S|$ & $43.52 (7.80)$ & $34.89 (7.34)$ & $\bf{0.10 (0.33)}$ & $2.95 (2.72)$\\ \hline

\multirow{3}{*}{ $h = 10$, $\bf{\sigma = 8}$, $\gamma = 1$} & $d(S(\hat x)|S_0)$ & $\bf{1.46 (1.28)}$ & $16.74 (34.27)$ & $19.92 (35.13)$ & $4.94 (3.39)$\\
 & $d(S_0|S(\hat x))$ & $75.44 (13.69)$ & $17.27 (16.89)$ & $\bf{11.07 (17.42)}$ & $19.95 (18.65)$ \\
 & $|S(\hat x)- S|$ & $44.37 (8.48)$ & $32.02 (7.54)$ & $\bf{0.31 (0.61)}$ & $2.43 (2.40)$\\ \hline

\multirow{3}{*}{ $h = 10, \sigma = 2$, $\bf{\gamma = 0.25}$} & $d(S(\hat x)|S_0)$ & $\bf{5.84 (4.23)}$ & $126.76 (72.29)$&  $128.71 (72.11)$ & $62.03 (62.73)$ \\
 & $d(S_0|S(\hat x))$ & $74.39 (15.97)$  & $16.95 (16.89)$ & $\bf{13.82 (17.78)}$ & $34.22 (22.65)$ \\
 & $|S(\hat x)- S|$ & $39.20 (10.32)$ & $11.68 (5.94)$ & $\bf{2.14 (1.16)}$  & $2.81 (3.01)$ \\ \hline

 \multirow{3}{*}{ $h = 10, \sigma = 2$, $\bf{\gamma = 0.5}$} & $d(S(\hat x)|S_0)$ & $\bf{1.51 (1.34)}$ & $17.36 (34.82)$ & $20.31 (35.51)$ & $6.29 (6.71)$\\
 & $d(S_0|S(\hat x))$ & $70.93 (13.86)$ & $15.47 (14.01)$ & $\bf{9.39 (15.57)}$ & $19.52 (19.52)$ \\
 & $|S(\hat x)- S|$ & $42.06 (10.29)$ & $29.73 (9.87)$ & $\bf{0.36 (0.75)}$& $2.34 (2.95)$\\ \hline

\multirow{3}{*}{ $h = 10, \sigma = 2$, $\bf{\gamma = 2}$} & $d(S(\hat x)|S_0)$ &$\bf{0.01 (0.10)}$ & $\bf{0.01 (0.10)}$ & $4.35 (6.44)$ & $0.46 (0.77)$\\
 & $d(S_0|S(\hat x))$ & $74.42 (14.35)$ & $17.96 (17.60)$ & $\bf{9.69 (19.38)}$ & $22.01 (25.52)$ \\
 & $|S(\hat x)- S|$ & $46.63 (10.19)$ & $46.42 (8.88)$ & $\bf{0.12 (0.48)}$& $3.76 (3.41)$\\

\end{tabular} 
\end{center}
\end{table}

\begin{table}[h]
\caption{Change point analysis results: band matrix and nine unequally spaced change points, with bandwidth $h \in \{ 1, 5, 10, 50\}$, noise level $\sigma \in \{0.5, 1, 2, 4\}$ and $\gamma \in \{0.25, 0.5, 1, 2\}$. The reported values are in the form of mean (standard deviation) across $100$ independent simulations. The bold-faced entries indicate the best method in different measures, with the convention that if $S(\hat x) = \emptyset$, then $d(S_0|S(\hat x)) = d(S_0|S(\hat x)) = p$.}\label{table_band_matrix_nine_unequal}
\begin{center}
\begin{tabular}{llllll}
   & & \textbf{FL} & \textbf{FLMF} & \textbf{FLMTF} & \textbf{ITALE}
\\ \hline
\multirow{3}{*}{ $\bf{h = 1}$, $\sigma = 1, \gamma = 1$} & $d(S(\hat x)|S_0)$ &  $\bf{2.13 (1.46)}$ & $31.01 (66.40)$& $ 40.68 (65.24)$ & $ 6.95 (6.72)$ \\
 & $d(S_0|S(\hat x))$ & $ 274.44 (32.13)$ & $ 21.77 (35.27)$ & $\bf{17.88 (34.28)}$ & $ 189.09 (97.11)$ \\
 & $|S(\hat x) - S|$ & $ 48.37 (11.52)$ & $ 30.73 (10.08)$ & $\bf{0.73 (0.75)}$  & $20.27 (18.31)$\\ \hline

\multirow{3}{*}{ $\bf{h = 5}$, $\sigma = 1,  \gamma = 1$} & $d(S(\hat x)|S_0)$ & $0.30 (0.58)$ & $ \bf{0.26 (0.51)}$ & $19.13 (8.54)$ & $3.03 (2.33)$ \\
 & $d(S_0|S(\hat x))$ & $ 278.84  (26.00)$ & $26.91 (50.28)$ & $\bf{25.95 (50.48)}$ & $96.28 (103.83)$ \\
 & $|S(\hat x) - S|$ & $50.71 (8.83)$ & $40.76 (7.55)$ & $\bf{1.02 (0.67)}$ & $9.48 (9.08)$\\ \hline

\multirow{3}{*}{ $\bf{h = 10}$, $\sigma = 1,  \gamma = 1$} & $d(S(\hat x)|S_0)$ & $0.09 (0.29)$ & $\bf{0.07 (0.26)}$ & $19.72 (10.15)$ & $1.12 (1.30)$\\
 & $d(S_0|S(\hat x))$ & $280.38 (22.33)$  & $31.59 (57.10)$ & $\bf{30.52 (57.50)}$ & $66.34 (92.42)$ \\
 & $|S(\hat x) - S|$ & $51.81 (10.03)$ & $44.08 (9.25)$ & $\bf{1.06 (0.76)}$ & $4.70 (4.28)$ \\ \hline

\multirow{3}{*}{ $\bf{h = 50}$, $\sigma = 1,  \gamma = 1$} & $d(S(\hat x)|S_0)$ &$\bf{0.00 (0.00)}$ & $\bf{0.00 (0.00)}$ & $ 21.42 (4.89) $ & $0.03 (0.17)$\\
 & $d(S_0|S(\hat x))$ & $267.51 (28.78)$ & $31.34 (59.91)$ & $32.07 (59.27)$ & $\bf{2.62 (9.38)}$ \\
 & $|S(\hat x) - S|$ & $49.10 (9.47)$ & $45.97 (8.95)$ & $1.31 (0.65)$ & $\bf{0.26 (0.63)}$\\ \hline

\multirow{3}{*}{ $h = 10$, $\bf{\sigma = 0.5}$, $\gamma = 1$} & $d(S(\hat x)|S_0)$ &  $\bf{0.00 (0.00)}$ & $\bf{0.00 (0.00)}$ & $20.33 (7.61)$ & $0.63 (1.05)$ \\
 & $d(S_0|S(\hat x))$ & $274.40 (26.76)$ & $ 38.37 (69.44)$ & $\bf{37.60 (69.58)}$ & $50.47 (78.13)$ \\
 & $|S(\hat x) - S|$ & $ 50.92 (10.23)$ & $44.56 (9.77)$ & $\bf{1.14 (0.64)}$  & $4.21 (3.90)$\\ \hline

\multirow{3}{*}{ $h = 10$, $\bf{\sigma = 2}$, $\gamma = 1$} & $d(S(\hat x)|S_0)$ & $0.29 (0.57)$ & $\bf{0.25 (0.50)}$ & $ 17.90 (7.84)$ & $ 2.03 (1.58)$\\
 & $d(S_0|S(\hat x))$ & $ 274.66 (26.76)$ & $ 27.48 (55.40)$ & $\bf{25.99 (55.62)}$ & $ 51.11 (82.25)$ \\
 & $|S(\hat x) - S|$ & $49.79 (11.00)$ & $39.88 (9.51)$ & $\bf{0.93 (0.62)}$ & $ 3.83 (3.87)$\\ \hline

\multirow{3}{*}{ $ h= 10$, $\bf{\sigma = 4}$, $\gamma = 1$} & $d(S(\hat x)|S_0)$ & $ 0.76 (0.84)$ & $ \bf{0.65 (0.82)}$ & $ 14.59 (9.27)$ & $4.10 (3.21)$\\
 & $d(S_0|S(\hat x))$ & $ 274.32 (25.61)$ & $ 37.71 (73.00)$ & $\bf{35.15 (73.93)}$ & $39.33 (67.42)$ \\
 & $|S(\hat x)- S|$ & $49.82 (9.54)$ & $ 35.92 (8.63)$ & $\bf{0.80 (0.78)}$ & $ 3.29 (4.06)$\\ \hline

\multirow{3}{*}{ $h = 10, \sigma = 1$, $\bf{\gamma = 0.25}$} & $d(S(\hat x)|S_0)$ & $\bf{3.21 (2.39)}$ & $ 92.24 (89.96)$&  $ 97.05 (89.28)$ & $ 50.65 (77.25)$ \\
 & $d(S_0|S(\hat x))$ & $ 274.59 (29.15)$  & $ 23.60 (37.97)$ & $\bf{19.17 (38.65)}$ & $ 53.25 (78.20)$ \\
 & $|S(\hat x)- S|$ & $ 44.51 (10.67)$ & $ 21.18 (7.91)$ & $\bf{1.12 (0.99)}$  & $2.26 (2.60)$ \\ \hline

 \multirow{3}{*}{ $h = 10, \sigma = 1$, $\bf{\gamma = 0.5}$} & $d(S(\hat x)|S_0)$ & $\bf{0.78 (1.05)}$ & $ 4.79 (28.05)$ & $ 19.25 (27.72)$ & $ 3.76 (2.99)$\\
 & $d(S_0|S(\hat x))$ & $ 278.48 (21.94)$ & $ 27.10 (45.32)$ & $\bf{23.66 (46.00)}$ & $33.96 (60.77)$ \\
 & $|S(\hat x)- S|$ & $ 50.98 (10.33)$ & $ 38.07 (9.60)$ & $\bf{0.66 (0.62)}$ & $2.85 (2.26)$\\ \hline

\multirow{3}{*}{ $h = 10, \sigma = 1$, $\bf{\gamma = 2}$} & $d(S(\hat x)|S_0)$ &$\bf{0.00 (0.00)}$ & $\bf{0.00 (0.00)}$ & $ 22.10 (7.59)$ & $0.32 (0.60)$\\
 & $d(S_0|S(\hat x))$ & $ 276.78 (26.56)$ & $ 28.49 (50.48)$ & $\bf{28.50 (50.22)}$ & $60.63 (88.71)$ \\
 & $|S(\hat x)- S|$ & $51.43 (9.15)$ & $46.91 (9.08)$ & $\bf{1.28 (0.62)}$& $4.54 (5.66)$\\ 

\end{tabular}
\end{center}
\end{table}

\subsection{Gaussian Random Matrix with Identity Covariance Matrix}

\Cref{table_random_matrix_one} displays the results of the first scenario of Gaussian random matrix with identity covariance matrix: only one change point with ratio $ n/p \in \{0.25, 0.5, 0.75, 1 \}$, noise level $\sigma \in \{1, 2, 4, 8\}$ and $\gamma \in \{0.25, 0.5, 1, 2\}$. \Cref{table_random_matrix_nine_equal} displays the results of the second scenario of Gaussian random matrix with identity covariance matrix:  $\left| S \right| = 9$  and change points are equally spaced, with ratio $ n/p \in \{0.25, 0.5, 0.75, 1 \}$, noise level $\sigma \in \{1, 2, 4, 8\}$ and $\gamma \in \{0.25, 0.5, 1, 2\}$. \Cref{table_random_matrix_nine_unequal} displays the results of the third scenario of Gaussian random matrix with identity covariance matrix:  $\left| S \right| = 9$  and change points are unequally spaced, with ratio $ n/p \in \{0.25, 0.5, 0.75, 1 \}$, noise level $\sigma \in \{0.5, 1, 2, 4\}$ and $\gamma \in \{0.25, 0.5, 1, 2\}$. 

\begin{table}[h]
\caption{Change point analysis results: Gaussian random matrix with identity covariance matrix and only one change point, with ratio $n/p \in \{ 0.25, 0.5, 0.75, 1\}$, noise level $\sigma \in \{1, 2, 4, 8\}$ and $\gamma \in \{0.25, 0.5, 1, 2\}$.  The reported values are in the form of mean (standard deviation) across $100$ independent simulations. The bold-faced entries indicate the best method in different measures, with the convention that if $S(\hat x) = \emptyset$, then $d(S_0|S(\hat x)) = d(S_0|S(\hat x)) = p$.}\label{table_random_matrix_one}
\begin{center}
\begin{tabular}{llllll}
   & & \textbf{FL} & \textbf{FLMF} & \textbf{FLMTF} & \textbf{ITALE}
\\ \hline

\multirow{3}{*}{ $\bf{n/p  = 0.25}$, $\sigma = 2,  \gamma = 1$} & $d(S(\hat x)|S_0)$ &  $\bf{0.00 (0.00)}$ & $\bf{0.00 (0.00)}$ & $ 2.08 (1.68)$ &  $\bf{0.00 (0.00)}$\\
 & $d(S_0|S(\hat x))$ & $ 263.81 (166.49)$ & $ 39.78 (104.58)$ & $ 32.85 (106.43)$ & $\bf{6.48 (49.56)}$ \\
 & $|S(\hat x) - S|$ & $ 8.40 (6.98)$ & $ 5.25 (3.92)$ & $0.13 (0.42)$  & $\bf{0.04 (0.24)}$\\ \hline

\multirow{3}{*}{ $\bf{n/p =0.5}$, $\sigma = 2, \gamma = 1$} & $d(S(\hat x)|S_0)$ & $\bf{0.00 (0.00)}$ & $\bf{0.00 (0.00)}$& $1.95 (1.90)$ & $\bf{0.00 (0.00)}$\\
 & $d(S_0|S(\hat x))$ & $254.03  (165.72)$ & $ 14.32 (16.68)$ & $\bf{6.59 (17.81)}$ & $ 5.27 (49.95)$ \\
 & $|S(\hat x) - S|$ & $8.35 (5.97)$ & $5.10 (3.15)$ &$\bf{0.08 (0.27)}$ & $0.09 (0.81)$ \\ \hline

\multirow{3}{*}{ $\bf{n/p = 0.75}$, $\sigma = 2,  \gamma = 1$} & $d(S(\hat x)|S_0)$ & $\bf{0.00 (0.00)}$ & $\bf{0.00 (0.00)}$ & $2.03 (1.96)$ & $\bf{0.00 (0.00)}$\\
 & $d(S_0|S(\hat x))$ & $261.30 (145.50)$  & $22.27 (67.39)$ & $\bf{13.98 (67.19)}$ & $19.12 (87.54)$ \\
 & $|S(\hat x) - S|$ & $8.27 (6.03)$ & $5.36 (2.85)$ & $\bf{0.05 (0.26)}$ & $ 0.14 (0.62)$ \\ \hline

\multirow{3}{*}{ $\bf{n/p = 1}$, $\sigma = 2,  \gamma = 1$} & $d(S(\hat x)|S_0)$ &$\bf{0.00 (0.00)}$ & $\bf{0.00 (0.00)}$ & $2.00 (2.07) $ & $\bf{0.00 (0.00)}$ \\
 & $d(S_0|S(\hat x))$ & $ 253.71 (146.34)$ & $ 32.40 (84.83)$ & $ 25.07 (85.93)$ & $\bf{14.01 (77.05)}$ \\
 & $|S(\hat x) - S|$ & $9.11 (7.97)$ & $5.02 (3.15)$ &  $0.13 (0.44)$ & $\bf{0.08 (0.39)}$ \\ \hline

\multirow{3}{*}{ $n/p =0.5$, $\bf{\sigma = 1}$, $\gamma = 1$} & $d(S(\hat x)|S_0)$ & $\bf{0.00 (0.00)}$ &$\bf{0.00 (0.00)}$ & $1.90 (2.23)$ & $\bf{0.00 (0.00)}$ \\
 & $d(S_0|S(\hat x))$ & $ 176.00 (153.44)$ & $ 11.20 (8.65)$ & $ 4.50 (10.38)$ & $\bf{0.00 (0.00)}$\\
 & $|S(\hat x) - S|$ & $ 5.50 (3.24)$ & $ 3.90 (2.56)$ &  $ 0.10 (0.32)$  & $\bf{0.00 (0.00)}$\\ \hline

\multirow{3}{*}{ $n/p =0.5$, $\bf{\sigma = 4}$, $\gamma = 1$} & $d(S(\hat x)|S_0)$ & $\bf{0.00 (0.00)}$ & $\bf{0.00 (0.00)}$ & $1.78 (1.69)$ & $\bf{0.00 (0.00)}$ \\
 & $d(S_0|S(\hat x))$ & $ 275.21 (155.44)$ & $ 25.82 (77.88)$ & $ 17.92 (78.88)$ & $\bf{11.50 (67.64)}$ \\
 & $|S(\hat x) - S|$ & $8.79 (5.75)$ & $ 5.64 (3.54)$ & $0.09 (0.38)$ & $\bf{0.07 (0.33)}$\\ \hline

\multirow{3}{*}{ $n/p =0.5$, $\bf{\sigma = 8}$, $\gamma = 1$} & $d(S(\hat x)|S_0)$ & $\bf{0.00 (0.00)}$ & $\bf{0.00 (0.00)}$ & $ 1.90 (1.83)$ & $\bf{0.00 (0.00)}$\\
 & $d(S_0|S(\hat x))$ & $253.01 (143.88)$ & $ 19.94 (61.84)$ & $\bf{11.74 (62.91)}$ &  $15.73 (66.81)$  \\
 & $|S(\hat x)- S|$ & $ 8.24 (5.16)$ & $ 5.35 (3.35)$ & $\bf{0.04 (0.24)}$ & $0.18 (0.66)$\\ \hline

\multirow{3}{*}{ $n/p =0.5, \sigma = 2$, $\bf{\gamma = 0.25}$} & $d(S(\hat x)|S_0)$ & $\bf{0.02 (0.14)}$ & $\bf{0.02 (0.14)}$ &  $ 1.86 (1.69)$ & $ 0.12 (0.38)$ \\
 & $d(S_0|S(\hat x))$ & $ 250.39 (150.71)$  & $ 29.20 (75.24)$ &  $\bf{22.19 (76.37)}$ & $22.47 (98.39)$ \\
 & $|S(\hat x)- S|$ & $8.72 (8.43$ & $5.66 (3.64)$ &  $0.18 (0.73)$ & $\bf{0.17 (0.88)}$ \\ \hline

 \multirow{3}{*}{ $n/p =0.5, \sigma = 2$, $\bf{\gamma = 0.5}$} & $d(S(\hat x)|S_0)$ & $\bf{0.00 (0.00)}$ & $\bf{0.00 (0.00)}$ & $1.97 (1.81)$ & $\bf{0.00 (0.00)}$\\
 & $d(S_0|S(\hat x))$ & $ 239.17 (149.06)$ & $ 28.60 (82.00)$ & $ 21.42 (83.60)$ & $\bf{10.83 (62.24)}$ \\
 & $|S(\hat x)- S|$ & $ 7.97 (5.74)$ & $5.08 (2.59)$ & $\bf{ 0.15 (0.58}$ & $0.14 (0.83)$ \\ \hline

\multirow{3}{*}{ $n/p =0.5,  \sigma = 2$, $\bf{\gamma = 2}$} & $d(S(\hat x)|S_0)$ & $\bf{0.00 (0.00)}$ & $\bf{0.00 (0.00)}$ & $ 2.07 (1.80)$ & $\bf{0.00 (0.00)}$\\
 & $d(S_0|S(\hat x))$ & $ 244.05 (157.22)$ & $ 28.08 (78.77)$ &  $20.01 (80.40)$ & $\bf{ 14.31 (81.83)}$\\
 & $|S(\hat x)- S|$ & $ 7.98 (5.06)$ & $5.44 (3.36)$ & $0.11 (0.42)$ & $\bf{0.05 (0.33)}$\\

\end{tabular}
\end{center}
\end{table}

\begin{table}[h] 
\caption{Change point analysis results: Gaussian random matrix with identity covariance matrix and nine equally spaced change points, with ratio $n/p \in \{ 0.25, 0.5, 0.75, 1\}$, noise level $\sigma \in \{1, 2, 4, 8\}$ and $\gamma \in \{0.25, 0.5, 1, 2\}$.  The reported values are in the form of mean (standard deviation) across $100$ independent simulations. The bold-faced entries indicate the best method in different measures, with the convention that if $S(\hat x) = \emptyset$, then $d(S_0|S(\hat x)) = d(S_0|S(\hat x)) = p$.}\label{table_random_matrix_nine_equal}
\begin{center}
\begin{tabular}{llllll}
   & & \textbf{FL} & \textbf{FLMF} & \textbf{FLMTF} & \textbf{ITALE}
\\ \hline

\multirow{3}{*}{ $\bf{n/p  = 0.25}$, $\sigma = 2,  \gamma = 1$} & $d(S(\hat x)|S_0)$ &  $\bf{0.00 (0.00)}$ & $\bf{0.00 (0.00)} $& $ 4.21 (2.24)$ &  $\bf{0.00 (0.00)}$\\
 & $d(S_0|S(\hat x))$ & $ 71.47 (15.35)$ & $ 19.79 (16.59)$ & $ 12.33 (19.35)$ & $\bf{1.99 (10.32)}$ \\
 & $|S(\hat x) - S|$ & $ 41.18 (6.25)$ & $46.32 (8.54)$ & $0.25 (0.59)$  & $\bf{0.16 (0.71)}$\\ \hline

\multirow{3}{*}{ $\bf{n/p =0.5}$, $\sigma = 2,   \gamma = 1$} & $d(S(\hat x)|S_0)$ & $\bf{0.00 (0.00)}$ & $\bf{0.00 (0.00)}$& $3.84 (1.59)$ & $\bf{0.00 (0.00)}$\\
 & $d(S_0|S(\hat x))$ & $ 73.67  (15.57)$ & $ 14.30 (8.56)$ & $ 6.08 (9.40)$ & $\bf{1.94 ( 10.60)}$ \\
 & $|S(\hat x) - S|$ & $42.69 (7.41)$ & $44.08 (8.63)$ & $\bf{0.06 (0.24)}$ & $0.13 (0.51)$ \\ \hline

\multirow{3}{*}{ $\bf{n/p = 0.75}$, $\sigma = 2,  \gamma = 1$} & $d(S(\hat x)|S_0)$ & $\bf{0.00 (0.00)}$& $\bf{0.00 (0.00)}$ & $3.81 (1.33)$ & $\bf{0.00 (0.00)}$\\
 & $d(S_0|S(\hat x))$ & $74.99 (15.22)$  & $16.23 (12.83)$ & $8.16 (14.38)$ & $\bf{0.66 (4.90)}$\\
 & $|S(\hat x) - S|$ & $43.53 (8.71)$ & $45.65 (9.90)$ & $ 0.13 (0.39)$ & $\bf{0.08 ( 0.46}$\\ \hline

\multirow{3}{*}{ $\bf{n/p = 1}$, $ \sigma = 2,  \gamma = 1$} & $d(S(\hat x)|S_0)$ &$\bf{0.00 (0.00)}$ & $\bf{0.00 (0.00)}$ & $3.84 (1.61)$ & $\bf{0.00 (0.00)}$ \\
 & $d(S_0|S(\hat x))$ & $74.77 (14.65)$ & $ 16.55 (13.95)$ & $8.61 (15.55)$ & $\bf{ 1.51 (10.49)}$ \\
 & $|S(\hat x) - S|$ & $43.66 (9.36)$ & $44.78 (8.60)$ &  $0.11 (0.35)$ & $\bf{0.05 ( 0.26)}$ \\ \hline

\multirow{3}{*}{ $n/p =0.5$, $\bf{\sigma = 1}$, $\gamma = 1$} & $d(S(\hat x)|S_0)$ & $\bf{0.00 (0.00)}$ & $\bf{0.00 (0.00)}$ & $4.22 (2.05)$ & $\bf{ 0.00 (0.00)}$ \\
 & $d(S_0|S(\hat x))$ & $75.45 (13.21)$ & $ 15.52 (9.26)$ & $ 6.75 (9.45)$ & $\bf{1.90 (11.35)}$\\
 & $|S(\hat x) - S|$ & $ 42.99 (8.33)$ & $44.77 (9.37)$ &  $ 0.09 (0.32)$  & $\bf{0.07 (0.38) }$\\ \hline

\multirow{3}{*}{ $n/p =0.5$, $\bf{\sigma = 4}$, $\gamma = 1$} & $d(S(\hat x)|S_0)$ & $\bf{0.00 (0.00)}$ & $\bf{0.00 (0.00)}$ & $ 4.01 (1.98)$ & $\bf{0.00 (0.00)}$ \\
 & $d(S_0|S(\hat x))$ & $74.61 (13.87)$ & $ 14.12 (7.91)$ & $6.35 (9.12)$ & $\bf{ 1.62 ( 10.22)}$ \\
 & $|S(\hat x) - S|$ & $43.56 (7.82)$ & $43.91 (8.64)$ & $\bf{0.08 (0.31)}$ & $ 0.11 (0.37)$\\ \hline

\multirow{3}{*}{ $n/p =0.5$, $\bf{\sigma = 8}$, $\gamma = 1$} & $d(S(\hat x)|S_0)$ & $\bf{0.00 (0.00)}$ & $\bf{0.00 (0.00) }$ & $ 4.16 (2.09)$ & $0.03 (0.17)$\\
 & $d(S_0|S(\hat x))$ & $ 72.57 (13.31)$ & $ 14.75 (11.13)$ &  $ 6.98 (12.05)$ & $\bf{ 0.51 (2.67)}$ \\
 & $|S(\hat x)- S|$ & $43.70 (8.00)$ & $ 45.01 (9.08)$ & $\bf{0.07 (0.26)}$ & $ 0.08 ( 0.31$\\ \hline

\multirow{3}{*}{ $n/p =0.5, \sigma = 2$, $\bf{\gamma = 0.25}$} & $d(S(\hat x)|S_0)$ & $\bf{ 0.06 ( 0.24)}$ & $\bf{ 0.06 (0.24)}$ &  $ 3.61 (1.45)$ & $0.40 (0.59)$ \\
 & $d(S_0|S(\hat x))$ & $ 72.22 (15.97) $  & $ 14.32 (10.37) $ &  $ 6.24 (11.34)$ & $\bf{1.62 (8.82)}$\\
 & $|S(\hat x)- S|$ & $ 43.37 (8.56)$ & $ 41.42 (9.38)$ &  $ 0.09 (0.40)$ & $\bf{ 0.06 (0.24)}$ \\ \hline

 \multirow{3}{*}{ $n/p =0.5,  \sigma = 2$, $\bf{\gamma = 0.5}$} & $d(S(\hat x)|S_0)$ & $\bf{0.00 (0.00)}$ & $\bf{ 0.00 (0.00)}$ & $ 4.16 (2.09)$ & $ 0.01 (0.10)$\\
 & $d(S_0|S(\hat x))$ & $ 77.67 (13.34)$ & $ 18.39 (15.92)$ & $ 10.24 (18.04)$ & $\bf{1.09 (8.75)}$ \\
 & $|S(\hat x)- S|$ & $ 43.50 (7.98)$ & $ 44.37 (8.46)$ & $0.16 (0.49)$ & $\bf{0.08 (0.44)}$ \\ \hline

\multirow{3}{*}{ $n/p =0.5, \sigma = 2$, $\bf{\gamma = 2}$} & $d(S(\hat x)|S_0)$ &$\bf{0.00 (0.00)}$ & $\bf{0.00 (0.00)} $& $ 3.72 (1.53) $ & $\bf{0.00 (0.00)}$\\
 & $d(S_0|S(\hat x))$ & $ 72.50 (15.31) $ & $ 16.59 (14.30)$ &  $8.83 (16.19) $ & $\bf{2.10 (13.14)}$\\
 & $|S(\hat x)- S|$ & $ 42.62 (7.46) $ & $45.37 (9.15) $ & $0.11 (0.31)$ &$\bf{0.07 (0.29)}$\\

\end{tabular}
\end{center}
\end{table}

\begin{table}[h]
\caption{Change point analysis results: Gaussian random matrix with identity covariance matrix and nine unequally spaced change points, with ratio $n/p \in \{ 0.25, 0.5, 0.75, 1\}$, noise level $\sigma \in \{0.5, 1, 2, 4\}$ and $\gamma \in \{0.25, 0.5, 1, 2\}$.  The reported values are in the form of mean (standard deviation) across $100$ independent simulations. The bold-faced entries indicate the best method in different measures, with the convention that if $S(\hat x) = \emptyset$, then $d(S_0|S(\hat x)) = d(S_0|S(\hat x)) = p$.}\label{table_random_matrix_nine_unequal}
\begin{center}
\begin{tabular}{llllll}
   & & \textbf{FL} & \textbf{FLMF} & \textbf{FLMTF} & \textbf{ITALE}
\\ \hline

\multirow{3}{*}{ $\bf{n/p  = 0.25}$, $\sigma = 1,  \gamma = 1$} & $d(S(\hat x)|S_0)$ &  $\bf{0.00 (0.00)}$ & $\bf{0.00 (0.00)}$& $20.34 (6.20)$ &  $\bf{ 0.00 (0.00) }$\\
 & $d(S_0|S(\hat x))$ & $ 257.33 (48.40)$ & $ 20.36 (29.78)$ & $ 19.44 (29.31)$ & $\bf{4.46 (29.08)}$ \\
 & $|S(\hat x) - S|$ & $ 38.81 (6.05)$ & $ 43.96 (8.52)$ & $1.19 (0.63)$  & $\bf{0.18 (0.78)}$\\ \hline

\multirow{3}{*}{ $\bf{n/p =0.5}$, $\sigma = 1,   \gamma = 1$} & $d(S(\hat x)|S_0)$ & $\bf{0.00 (0.00)}$ & $\bf{ 0.00 (0.00)}$& $20.40 (6.09)$ & $\bf{0.00 (0.00)}$\\
 & $d(S_0|S(\hat x))$ & $ 268.07  (36.07)$ & $ 20.54 (29.75)$ & $ 19.81 (30.04)$ & $\bf{3.25 (24.96)}$ \\
 & $|S(\hat x) - S|$ & $43.56 (7.09)$ & $44.47 (8.75)$ & $1.14 (0.64)$ & $\bf{0.06 (0.28)}$ \\ \hline

\multirow{3}{*}{ $\bf{n/p = 0.75}$, $\sigma = 1,  \gamma = 1$} & $d(S(\hat x)|S_0)$ & $\bf{0.00 (0.00)}$ & $\bf{0.00 (0.00)}$ & $20.76 (5.74)$ & $\bf{0.00 (0.00)}$\\
 & $d(S_0|S(\hat x))$ & $276.55 (23.80)$  & $ 28.37 (55.20)$ & $ 28.56 (54.91)$ & $\bf{1.80 (11.86)}$\\
 & $|S(\hat x) - S|$ & $45.39 (7.32)$ & $43.66 (8.95)$ & $ 1.19 (0.63)$ & $\bf{0.08 (0.34)}$\\ \hline

\multirow{3}{*}{ $\bf{n/p = 1}$, $\sigma = 1$, $\gamma = 1$} & $d(S(\hat x)|S_0)$ & $\bf{0.00 (0.00)}$ &$\bf{ 0.00 (0.00)}$ & $20.32 (6.22)$ & $\bf{ 0.00 (0.00)}$ \\
 & $d(S_0|S(\hat x))$ & $ 272.71 (27.03) $ & $ 21.45 (41.23) $ & $ 21.16 (41.22) $ & $\bf{1.70 (9.60}$\\
 & $|S(\hat x) - S|$ & $ 46.37 (8.06) $ & $44.05  (7.30) $ &  $ 1.24 (0.59) $  & $\bf{0.10 (0.48)}$\\ \hline
 
 \multirow{3}{*}{ $n/p = 0.5$, $\bf{\sigma = 0.5}$, $\gamma = 1$} & $d(S(\hat x)|S_0)$ &$\bf{0.00 (0.00)}$ & $\bf{0.00 (0.00)}$ & $ 20.88 (7.66) $ & $\bf{0.00 (0.00)}$ \\
 & $d(S_0|S(\hat x))$ & $271.82 (29.89) $ & $ 18.21 (29.21)$ & $ 18.68 (28.96)$ & $\bf{ 0.97 (6.08)}$ \\
 & $|S(\hat x) - S|$ & $43.88 (7.36)$ & $43.33 (9.04)$ &  $1.24 (0.64)$ & $\bf{0.11 (0.65)}$ \\ \hline

\multirow{3}{*}{ $n/p =0.5$, $\bf{\sigma = 2}$, $\gamma = 1$} & $d(S(\hat x)|S_0)$ & $\bf{0.00 (0.00)}$ & $\bf{0.00 (0.00)}$ & $ 20.72 (5.28)$ & $\bf{ 0.00 (0.00)}$ \\
 & $d(S_0|S(\hat x))$ & $ 275.41 (26.15)$ & $ 20.03 (30.01)$ & $ 19.85 (29.78)$ & $\bf{ 1.13 (4.89)}$ \\
 & $|S(\hat x) - S|$ & $44.61  (7.08)$ & $42.85 (7.98)$ & $1.21 (0.62)$ & $\bf{ 0.10 (0.44)}$ \\ \hline

\multirow{3}{*}{ $n/p =0.5$, $\bf{\sigma = 4}$, $\gamma = 1$} & $d(S(\hat x)|S_0)$ & $\bf{0.00 (0.00)}$ & $\bf{ 0.00 (0.00)}$ & $ 20.33 (5.74)$ & $\bf{0.00 (0.00)}$\\
 & $d(S_0|S(\hat x))$ & $ 272.91 (29.13)$ & $ 29.29 (54.87)$ &  $ 28.87 (54.94)$ & $\bf{ 3.57 (28.91)}$ \\
 & $|S(\hat x)- S|$ & $45.27 (6.87) $ & $ 44.15 (7.91) $ &  $1.12 (0.62)$ & $\bf{ 0.09 (0.35)}$ \\ \hline

\multirow{3}{*}{ $n/p =0.5, \sigma = 1$, $\bf{\gamma = 0.25}$} & $d(S(\hat x)|S_0)$ & $\bf{ 0.00 (0.00)}$ & $\bf{ 0.00 (0.00)}$ &  $21.40 (6.22)$ & $ 0.09 (0.32)$ \\
 & $d(S_0|S(\hat x))$ & $ 274.74 (25.66)$  & $ 30.19 (43.54)$ &  $ 29.66 (43.45) $ & $\bf{ 1.02 (4.77)}$\\
 & $|S(\hat x)- S|$ & $ 44.92 (6.46)$ & $ 43.97 (8.92)$ &  $ 1.08 (0.68)$ & $\bf{ 0.13 (0.39)}$ \\ \hline

 \multirow{3}{*}{ $n/p =0.5, \sigma = 1$, $\bf{\gamma = 0.5}$} & $d(S(\hat x)|S_0)$ & $\bf{0.00 (0.00)}$ & $\bf{ 0.00 (0.00)}$ & $ 21.51 (7.52)$ & $\bf{ 0.00 (0.00)}$\\
 & $d(S_0|S(\hat x))$ & $ 271.97  (27.29)$ & $ 34.19 (62.37)$ & $ 33.86 (61.92)$ & $\bf{0.23 (1.68)}$ \\
 & $|S(\hat x)- S|$ & $ 44.80 (8.18)$ & $ 43.52 (8.68)$ & $1.13 (0.65)$ & $\bf{0.03 (0.17)}$ \\ \hline

\multirow{3}{*}{ $n/p =0.5, \sigma = 1$, $\bf{\gamma = 2}$} & $d(S(\hat x)|S_0)$ & $\bf{ 0.00 (0.00)}$ & $\bf{ 0.00 (0.00)}$ & $20.55 (7.28)$ & $\bf{0.00 (0.00)}$\\
 & $d(S_0|S(\hat x))$ & $ 270.47 (29.05)$ & $ 25.82 (48.23)$ &  $ 25.23 (48.06)$ & $\bf{1.82  (13.22)}$\\
 & $|S(\hat x)- S|$ & $ 43.20 (7.54)$ & $43.08 (8.03)$ & $ 1.15 (0.63)$ & $\bf{ 0.07 (0.36)}$\\

\end{tabular}
\end{center}
\end{table}

\subsection{Gaussian Random Matrix with Band Covariance Matrix}

\Cref{table_random_matrix_band_covariance_one} displays the results of the first scenario of Gaussian random matrix with band covariance matrix: only one change point, with bandwidth $ h \in \{1, 10, 50, 999 \}$, ratio $n/p \in \{0.25, 0.5, 0.75, 1 \}$ and noise level $\sigma \in \{1, 2, 4, 8\}$. \Cref{table_random_matrix_band_covariance_nine_equal} displays the result of the second scenario of Gaussian random matrix with band covariance matrix:  $\left| S \right| = 9$  and change points are equally spaced for bandwidth $ h \in \{1, 10, 50, 999 \}$, ratio $n/p \in \{0.25, 0.5, 0.75, 1 \}$ and noise level $\sigma \in \{1, 2, 4, 8\}$.  \Cref{table_random_matrix_band_covariance_nine_unequal} displays the result of the third scenario of Gaussian random matrix with band covariance matrix:  $\left| S \right| = 9$  and change points are unequally spaced for bandwidth $ h \in \{1, 10, 50, 999 \}$, ratio $n/p \in \{0.25, 0.5, 0.75, 1 \}$ and noise level $\sigma \in \{0.5, 2, 4\}$.

\begin{table}[h]
\caption{Change point analysis results: Gaussian random matrix with band covariance matrix and only one change point, with bandwidth $ h \in \{1, 10, 50, 999 \}$, ratio $n/p \in \{0.25, 0.5, 0.75, 1 \}$ and noise level $\sigma \in \{1, 2, 4, 8\}$. The reported values are in the form of mean (standard deviation) across $100$ independent simulations. The bold-faced entries indicate the best method in different measures, with the convention that if $S(\hat x) = \emptyset$, then $d(S_0|S(\hat x)) = d(S_0|S(\hat x)) = p$.}\label{table_random_matrix_band_covariance_one}
\begin{center}
\begin{tabular}{llllll}
   & & \textbf{FL} & \textbf{FLMF} & \textbf{FLMTF} & \textbf{ITALE}
\\ \hline

\multirow{3}{*}{ $\bf{h = 1}$, $n/p  = 0.5, \sigma = 2$} & $d(S(\hat x)|S_0)$ &  $\bf{0.00 (0.00)}$ & $ \bf{0.00 (0.00)}$ & $ 1.82 (1.62)$ &  $\bf{ 0.00 (0.00)}$\\
 & $d(S_0|S(\hat x))$ & $236.56 (162.54)$ & $26.38 (82.93)$ & $\bf{ 18.75 (84.41)}$ & $67.21 (154.40)$ \\
 & $|S(\hat x) - S|$ & $ 6.55 (5.01)$ & $ 4.54 (2.86)$ & $\bf{ 0.05 (0.26)}$  & $0.35 (0.96)$ \\ \hline

\multirow{3}{*}{ $\bf{h = 10}$, $n/p =0.5, \sigma = 2$} & $d(S(\hat x)|S_0)$ & $\bf{0.00 (0.00)}$ & $\bf{0.00 (0.00)}$ & $1.82 (1.90)$ & $\bf{ 0.00 (0.00)}$\\
 & $d(S_0|S(\hat x))$ & $ 231.10  (154.62)$ & $ 25.20 (76.62)$ & $\bf{17.27 (77.78)}$ & $62.97 (144.90)$ \\
 & $|S(\hat x) - S|$ & $ 6.80 (4.64)$ & $5.24 (2.66)$ & $\bf{0.09 (0.35)}$ & $0.52 (1.30)$ \\ \hline

\multirow{3}{*}{ $\bf{h = 50}$, $n/p = 0.5, \sigma = 2$} & $d(S(\hat x)|S_0)$ & $\bf{ 0.00 (0.00)}$& $\bf{ 0.00 (0.00)}$ & $2.23 (1.87)$ & $\bf{ 0.00 (0.00)}$\\
 & $d(S_0|S(\hat x))$ & $228.29 (168.52)$  & $39.60 (102.42)$ & $\bf{31.67 (104.58)}$ & $ 59.36 (138.74)$ \\
 & $|S(\hat x) - S|$ & $ 7.38 (7.37)$ & $5.71 (3.14)$ & $\bf{ 0.18 (0.66)}$ & $0.39 (1.00)$ \\ \hline

\multirow{3}{*}{ $\bf{h = 999}$, $n/p = 0.5, \sigma = 2$} & $d(S(\hat x)|S_0)$ &$\bf{0.00 (0.00)}$ & $\bf{0.00 (0.00)}$ & $ 1.74 (1.50)$ & $\bf{0.00 (0.00)}$ \\
 & $d(S_0|S(\hat x))$ & $ 249.63 (166.74)$ & $ 22.67 (59.67)$ & $\bf{15.35 (60.78)}$ & $ 53.54 (134.19)$ \\
 & $|S(\hat x) - S|$ & $ 6.93 (5.35)$ & $ 4.85 (3.39)$ & $\bf{ 0.12 (0.46)}$ & $0.43 (1.18)$ \\ \hline

\multirow{3}{*}{  $h = 10$, $\bf{n/p =0.25}$, $\sigma = 2$} & $d(S(\hat x)|S_0)$ & $\bf{0.00 (0.00)}$ & $\bf{ 0.00 (0.00)}$& $2.60 (2.82)$ & $\bf{ 0.00 (0.00)}$ \\
 & $d(S_0|S(\hat x))$ & $ 277.87 (150.62)$ & $ 24.28 (68.39)$ & $\bf{ 16.79 (69.51)}$ & $ 51.66 (129.12)$\\
 & $|S(\hat x) - S|$ & $ 7.71 (5.92)$ & $ 4.75 (2.96)$ & $\bf{ 0.12 (0.54)}$ & $0.44 (1.21)$ \\ \hline

\multirow{3}{*}{ $h = 10$, $\bf{n/p =0.75}$, $\sigma = 2$} & $d(S(\hat x)|S_0)$ & $\bf{ 0.00 (0.00)}$ & $\bf{0.00 (0.00)}$ & $ 1.90 (1.91)$ & $\bf{ 0.00 (0.00)}$ \\
 & $d(S_0|S(\hat x))$ & $ 199.22 (156.85)$ & $ 17.14 (48.55)$ & $\bf{ 8.73 (48.72)}$ & $ 51.54 (133.34)$ \\
 & $|S(\hat x) - S|$ & $5.64 (4.75)$ & $ 5.25 (3.12)$ & $\bf{0.04 (0.24)}$ & $0.38 (0.91)$\\ \hline

\multirow{3}{*}{ $h = 10$, $\bf{n/p = 1}$, $\sigma = 2$} & $d(S(\hat x)|S_0)$ & $\bf{0.00 (0.00)}$ & $\bf{ 0.00 (0.00)}$ & $ 1.80 (1.82)$ & $\bf{0.00 (0.00)}$\\
 & $d(S_0|S(\hat x))$ & $ 263.41 (164.20)$ & $ 17.66 (57.97)$ & $\bf{ 10.26 (58.90)}$ &  $65.98 (138.94)$  \\
 & $|S(\hat x)- S|$ & $ 6.60 (5.37)$ & $ 4.61 (2.88)$ & $\bf{ 0.04 (0.24)}$ & $ 0.50 (1.14)$\\ \hline

\multirow{3}{*}{ $h = 10, n/p =0.5$, $\bf{\sigma = 1}$} & $d(S(\hat x)|S_0)$ & $\bf{ 0.00 (0.00)}$ & $\bf{ 0.00 (0.00)}$ &  $ 1.89 (1.75)$ & $\bf{0.00 (0.00)}$ \\
 & $d(S_0|S(\hat x))$ & $ 221.62 (160.88)$  & $ 21.51 (68.94)$ &  $\bf{14.16 (69.94)}$ & $ 82.38 (162.18)$ \\
 & $|S(\hat x)- S|$ & $ 5.43 (3.33) $ & $ 4.59 (2.72) $ & $\bf{ 0.06 (0.34)}$&  $0.68 (1.56)$ \\ \hline

 \multirow{3}{*}{ $h = 10, n/p =0.5$, $\bf{\sigma = 4}$} & $d(S(\hat x)|S_0)$ & $\bf{0.00 (0.00)}$ & $\bf{ 0.00 (0.00)}$ & $ 1.98 (2.19)$ & $\bf{0.00 (0.00)}$\\
 & $d(S_0|S(\hat x))$ & $ 247.59 (162.44)$ & $ 16.20 (38.83)$ & $\bf{8.00 (39.33)}$ & $32.49 (104.56)$ \\
 & $|S(\hat x)- S|$ & $ 7.37 (5.86) $ & $ 4.83 (2.76)$ & $\bf{ 0.06 (0.34)}$ & $0.26 (0.91)$ \\ \hline

\multirow{3}{*}{ $h = 10, n/p =0.5$, $\bf{\sigma = 8}$} & $d(S(\hat x)|S_0)$ & $\bf{ 0.00 (0.00)}$ & $\bf{ 0.00 (0.00)}$ & $1.99 (1.55)$ & $\bf{0.00 (0.00)}$\\
 & $d(S_0|S(\hat x))$ & $ 248.91 (148.06)$ & $ 20.73 (53.63)$ &  $\bf{13.37 (54.97)}$ & $ 20.13 (85.08)$ \\
 & $|S(\hat x)- S|$ & $ 6.55 (3.71)$ & $ 4.61 (2.69)$ & $\bf{0.07 (0.29)}$ &  $0.11 (0.42)$ \\

\end{tabular}
\end{center}
\end{table}

\begin{table}[h] 
\caption{Change point analysis results: Gaussian random matrix with band covariance matrix and nine equally spaced change points, with bandwidth $ h \in \{1, 10, 50, 999 \}$, ratio $n/p \in \{0.25, 0.5, 0.75, 1 \}$ and noise level $\sigma \in \{1, 2, 4, 8\}$.  The reported values are in the form of mean (standard deviation) across $100$ independent simulations. The bold-faced entries indicate the best method in different measures, with the convention that if $S(\hat x) = \emptyset$, then $d(S_0|S(\hat x)) = d(S_0|S(\hat x)) = p$.}\label{table_random_matrix_band_covariance_nine_equal}
\begin{center}
\begin{tabular}{llllll}
   & & \textbf{FL} & \textbf{FLMF} & \textbf{FLMTF} & \textbf{ITALE}
\\ \hline

\multirow{3}{*}{ $\bf{h = 1}$, $n/p  = 0.5, \sigma = 2$} & $d(S(\hat x)|S_0)$ &  $\bf{0.00 (0.00)}$ & $ \bf{0.00 (0.00)}$ & $ 4.37 (2.73)$ &  $\bf{ 0.00 (0.00)}$\\
 & $d(S_0|S(\hat x))$ & $ 73.82 (16.05)$ & $ 18.50 (16.45)$ & $\bf{10.39 (18.19)}$ & $ 30.24 (33.39)$ \\
 & $|S(\hat x) - S|$ & $ 38.94 (7.64)$ & $ 43.91 (9.55)$ & $\bf{0.12 (0.36)}$  & $3.18 (4.91)$ \\ \hline

\multirow{3}{*}{ $\bf{h = 10}$, $n/p =0.5, \sigma = 2$} & $d(S(\hat x)|S_0)$ & $\bf{0.00 (0.00)}$ & $\bf{ 0.00 (0.00)}$& $3.91 (1.90)$ & $\bf{ 0.00 (0.00)}$\\
 & $d(S_0|S(\hat x))$ & $ 73.17  (14.27)$ & $ 15.37 (11.06)$ & $\bf{6.40 (11.37)}$ & $ 20.68 (26.00)$ \\
 & $|S(\hat x) - S|$ & $ 37.97 (6.73)$ & $ 42.02 (8.08)$ & $\bf{0.05 (0.22)}$& $2.30 (4.74)$ \\ \hline

\multirow{3}{*}{ $\bf{h = 50}$, $n/p = 0.5, \sigma = 2$} & $d(S(\hat x)|S_0)$ & $\bf{ 0.00 (0.00)}$ & $\bf{ 0.00 (0.00)}$ & $3.83 (1.81)$ & $\bf{0.00 (0.00)}$\\
 & $d(S_0|S(\hat x))$ & $72.12 (16.77)$  & $ 16.01 (12.28)$ & $\bf{7.81 (13.40)}$ & $ 30.75 (29.63)$ \\
 & $|S(\hat x) - S|$ & $ 38.13 (6.73)$ & $ 42.89 (8.05)$ & $\bf{ 0.12 (0.38)}$ & $ 2.94 (3.88)$ \\ \hline

\multirow{3}{*}{ $\bf{h = 999}$, $n/p = 0.5, \sigma = 2$} & $d(S(\hat x)|S_0)$ & $\bf{0.00 (0.00)}$ & $\bf{0.00 (0.00)}$ & $ 3.70 (1.56) $ & $\bf{0.00 (0.00)}$ \\
 & $d(S_0|S(\hat x))$ & $ 69.41 (16.69)$ & $ 13.20 (9.66)$ & $\bf{5.29 (10.14)}$ &  $ 28.28 (30.02)$ \\
 & $|S(\hat x) - S|$ & $ 37.83 (7.28)$ & $ 42.42 (10.27)$ & $\bf{ 0.04 (0.24)}$ & $2.76 (3.91)$ \\ \hline

\multirow{3}{*}{  $h = 10$, $\bf{n/p =0.25}$, $\sigma = 2$} & $d(S(\hat x)|S_0)$ & $\bf{0.00 (0.00)}$ & $\bf{0.00 (0.00)}$& $ 4.15 (1.86)$ & $\bf{ 0.00 (0.00)}$ \\
 & $d(S_0|S(\hat x))$ & $ 71.77 (15.53)$ & $ 17.20 (12.35)$ & $\bf{ 8.82 (14.01)}$ & $ 17.29 (27.19)$\\
 & $|S(\hat x) - S|$ & $ 37.19 (5.66)$ & $ 41.72 (8.76)$ &  $\bf{0.16 (0.48)}$  & $1.17 (2.67)$ \\ \hline

\multirow{3}{*}{ $h = 10$, $\bf{n/p =0.75}$, $\sigma = 2$} & $d(S(\hat x)|S_0)$ & $\bf{0.00 (0.00)}$ &$\bf{0.00 (0.00)}$ & $ 3.65 (1.31)$ & $\bf{ 0.00 (0.00)}$ \\
 & $d(S_0|S(\hat x))$ & $ 71.85 (15.90)$ & $ 17.37 (16.16)$ & $\bf{9.98 (18.09)}$ & $ 27.00 (29.24)$ \\
 & $|S(\hat x) - S|$ & $36.53 (7.55)$ & $ 41.20 (7.43)$ & $\bf{0.13 (0.34)}$ & $ 2.70 (3.83)$\\ \hline

\multirow{3}{*}{ $h = 10$, $\bf{n/p = 1}$, $\sigma = 2$} & $d(S(\hat x)|S_0)$ & $\bf{0.00 (0.00)}$ & $\bf{ 0.00 (0.00)}$ & $ 3.95 (1.71)$ & $\bf{0.00 (0.00)}$\\
 & $d(S_0|S(\hat x))$ & $ 73.78 (15.50)$ & $ 16.63 (14.67)$ & $\bf{8.93 (16.32)}$ &  $ 25.51 (35.16)$  \\
 & $|S(\hat x)- S|$ & $ 37.64 (8.48)$ & $ 41.44 (8.62)$ & $\bf{ 0.11 (0.35)}$ & $2.04 (3.22)$\\ \hline

\multirow{3}{*}{ $h = 10, n/p =0.5$, $\bf{\sigma = 1}$} & $d(S(\hat x)|S_0)$ & $\bf{ 0.00 (0.00)}$ & $\bf{ 0.00 (0.00)}$ &  $3.70 (1.18)$ & $ 0.12 (0.38)$ \\
 & $d(S_0|S(\hat x))$ & $71.25 (17.26)$  & $15.28 (10.57)$ &  $\bf{7.85 (12.68)}$ & $43.38 (32.25)$ \\
 & $|S(\hat x)- S|$ & $ 35.43 (7.33)$ & $ 39.45 (7.98)$ &  $\bf{0.12 (0.36)}$ &  $ 5.66 (6.30)$ \\ \hline

 \multirow{3}{*}{ $h = 10, n/p =0.5$, $\bf{\sigma = 4}$} & $d(S(\hat x)|S_0)$ & $\bf{0.00 (0.00)}$ & $\bf{ 0.00 (0.00)}$ & $ 4.14 (2.07)$ & $\bf{0.00 (0.00)}$\\
 & $d(S_0|S(\hat x))$ & $ 73.90 (14.90)$ & $ 18.76 (17.83)$ & $\bf{10.67 (19.76)}$& $13.74 (25.15)$ \\
 & $|S(\hat x)- S|$ & $41.53 (7.68)$ & $ 44.69 (7.82)$ & $\bf{0.15 (0.46)}$ & $0.82 (1.62)$ \\ \hline

\multirow{3}{*}{ $h = 10, n/p =0.5$, $\bf{\sigma = 8}$} & $d(S(\hat x)|S_0)$ & $\bf{ 0.00 (0.00)}$ & $\bf{ 0.00 (0.00)}$ & $ 3.89 (1.80)$ & $\bf{0.00 (0.00)}$\\
 & $d(S_0|S(\hat x))$ & $ 75.88 (15.43)$ & $19.07 (17.77)$ &  $11.76 (20.12)$ & $\bf{8.49 (18.94)}$\\
 & $|S(\hat x)- S|$ & $ 42.38 (7.56)$ & $44.95 (8.43)$ & $\bf{ 0.19 (0.49)}$ &  $0.45 (0.98)$ \\

\end{tabular}
\end{center}
\end{table}

\begin{table}[h]
\caption{Change point analysis results: Gaussian random matrix with band covariance matrix and nine unequally spaced change points, with bandwidth $ h \in \{1, 10, 50, 999 \}$, ratio $n/p \in \{0.25, 0.5, 0.75, 1 \}$ and noise level $\sigma \in \{0.5, 2, 4\}$. The reported values are in the form of mean (standard deviation) across $100$ independent simulations. The bold-faced entries indicate the best method in different measures, with the convention that if $S(\hat x) = \emptyset$, then $d(S_0|S(\hat x)) = d(S_0|S(\hat x)) = p$.}\label{table_random_matrix_band_covariance_nine_unequal}
\begin{center}
\begin{tabular}{llllll}
   & & \textbf{FL} & \textbf{FLMF} & \textbf{FLMTF} & \textbf{ITALE}
\\ \hline

\multirow{3}{*}{ $\bf{h = 1}$, $n/p  = 0.5, \sigma = 2$} & $d(S(\hat x)|S_0)$ &  $\bf{0.00 (0.00)}$ & $\bf{0.00 (0.00)}$ & $ 20.44 (7.57)$ &  $\bf{ 0.00 (0.00)}$\\
 & $d(S_0|S(\hat x))$ & $ 260.78 (43.80)$ & $ 25.68 (45.75)$ & $\bf{25.58 (45.45)}$ & $ 36.99 (73.95)$ \\
 & $|S(\hat x) - S|$ & $ 39.67 (6.99)$ & $ 43.27 ( 8.89$ & $\bf{1.15 (0.66)}$  & $1.83 (3.15)$ \\ \hline

\multirow{3}{*}{$\bf{h = 10}$, $n/p  = 0.5, \sigma = 2$} & $d(S(\hat x)|S_0)$ & $\bf{0.00 (0.00)}$ & $\bf{ 0.00 (0.00)}$ & $20.27 (5.01)$ & $\bf{ 0.00 (0.00) }$\\
 & $d(S_0|S(\hat x))$ & $ 264.26  (39.21)$ & $ 30.17 (57.87)$ & $\bf{ 29.66 (57.77)}$ & $43.60 (86.91)$ \\
 & $|S(\hat x) - S|$ & $37.39 (5.59)$ & $39.99 (7.79)$ & $\bf{1.09 (0.59)}$ & $1.61 (3.29)$ \\ \hline

\multirow{3}{*}{ $\bf{h = 50}$, $n/p  = 0.5, \sigma = 2$} & $d(S(\hat x)|S_0)$ & $\bf{ 0.00 (0.00)}$ & $\bf{ 0.00 (0.00)}$ & $ 20.07 (7.08)$ & $\bf{0.00 (0.00)}$\\
 & $d(S_0|S(\hat x))$ & $ 264.38 (38.43)$  & $ 20.20 (36.47)$ & $\bf{19.91 (36.29)}$ & $37.70 (80.75)$ \\
 & $|S(\hat x) - S|$ & $37.45 (6.28)$ & $40.92  (7.88)$ & $\bf{ 1.13 (0.66)}$ & $ 1.52 (2.61)$ \\ \hline

\multirow{3}{*}{$\bf{h = 999}$, $n/p  = 0.5, \sigma = 2$} & $d(S(\hat x)|S_0)$ & $\bf{  0.00 (0.00)}$ & $\bf{ 0.00 (0.00)}$ & $19.49 (6.45)$ & $\bf{ 0.00 (0.00)}$ \\
 & $d(S_0|S(\hat x))$ & $246.50 (55.18)$ & $21.87 (42.64)$ & $\bf{21.43 (42.80)}$ & $40.96 (78.84)$\\
 & $|S(\hat x) - S|$ & $ 32.57 (6.19)$ & $39.90 (8.06)$ & $\bf{ 1.13 (0.65)}$  & $1.85 (3.53)$ \\ \hline
 
 \multirow{3}{*}{ $h = 10$, $\bf{n/p = 0.25}$, $\sigma = 2$} & $d(S(\hat x)|S_0)$ & $\bf{0.00 (0.00)}$ & $\bf{0.00 (0.00)}$ & $20.22 (7.82)$ & $\bf{0.00 (0.00)}$ \\
 & $d(S_0|S(\hat x))$ & $ 244.04 (55.81)$ & $38.07 (71.53)$ & $\bf{ 36.84 (71.74)}$ & $36.05 (73.09)$ \\
 & $|S(\hat x) - S|$ & $35.08 (6.43)$ & $40.46 (8.58)$ & $\bf{1.09 (0.65)}$ & $1.37 (2.27)$ \\ \hline

\multirow{3}{*}{ \bf{$h = 10$}, $\bf{n/p  = 0.75}$, $\sigma = 2$} & $d(S(\hat x)|S_0)$ & $\bf{ 0.00 (0.00)}$ & $\bf{0.00 (0.00)}$ & $ 20.87 (6.29)$ & $\bf{ 0.00 (0.00) }$ \\
 & $d(S_0|S(\hat x))$ & $ 266.44 (35.89)$ & $ 21.95 (43.96)$ & $\bf{ 21.31 (44.07)}$ & $35.51 (72.03)$ \\
 & $|S(\hat x) - S|$ & $38.93  (7.09)$ & $39.84 (7.92)$ & $\bf{1.21 (0.59)}$ &  $2.09 (4.11)$ \\ \hline

\multirow{3}{*}{\bf{$h = 10$}, $\bf{n/p  = 1}$, $\sigma = 2$} & $d(S(\hat x)|S_0)$ & $\bf{ 0.00 (0.00)}$ & $\bf{ 0.00 (0.00)}$ & $ 19.52 (6.13)$ & $\bf{0.00 (0.00)}$\\
 & $d(S_0|S(\hat x))$ & $262.17 (39.43)$ & $\bf{28.42 (54.78)}$ &  $ 28.63 (54.81) $ & $31.62 (69.25)$  \\
 & $|S(\hat x)- S|$ & $37.69 (7.69)$ & $ 40.40 (7.57)$ & $\bf{1.11 (0.65)}$ & $ 1.38 (2.30)$ \\ \hline

\multirow{3}{*}{ $h = 10, n/p =0.5$, $\bf{\sigma = 0.5}$ } & $d(S(\hat x)|S_0)$ & $\bf{ 0.00 (0.00)}$ & $\bf{ 0.00 (0.00)}$ &  $19.75 (5.95)$ & $ 0.09 (0.32)$ \\
 & $d(S_0|S(\hat x))$ & $ 252.26 (46.95)$  & $\bf{19.17 (33.98)}$ &  $19.99 (33.71)$ & $62.16 (91.68)$ \\
 & $|S(\hat x)- S|$ & $ 36.47 (5.99)$ & $ 40.36 (8.11)$ & $\bf{ 1.21 ( 0.62)}$ &  $3.46 (5.08)$ \\ \hline

 \multirow{3}{*}{ $h = 10, n/p =0.5$, $\bf{\sigma = 2}$} & $d(S(\hat x)|S_0)$ & $\bf{0.00 (0.00)}$ & $\bf{ 0.00 (0.00)}$ & $20.88 (6.45)$ & $\bf{0.00 (0.00)}$\\
 & $d(S_0|S(\hat x))$ & $ 260.71  (41.38)$ & $25.16 (45.30)$ & $ 24.92 (45.04)$ & $\bf{14.56 (44.48)}$ \\
 & $|S(\hat x)- S|$ & $ 39.69 (6.60)$ & $ 41.17 (7.75)$ & $1.12 (0.61)$ & $\bf{0.56 (1.36)}$ \\ \hline

\multirow{3}{*}{ $h = 10, n/p =0.5$, $\bf{\sigma = 4}$} & $d(S(\hat x)|S_0)$ & $\bf{0.00 (0.00)}$ & $\bf{0.00 (0.00)}$ & $ 19.79 (7.05)$ & $\bf{0.00 (0.00)}$\\
 & $d(S_0|S(\hat x))$ & $ 272.77 (29.71)$ & $ 24.01 (49.14)$ &  $ 23.01 (48.83)$ & $\bf{7.47  (37.18)}$\\
 & $|S(\hat x)- S|$ & $ 41.35 (7.59)$ & $40.97 (9.21)$ & $ 1.15 (0.67)$ & $\bf{ 0.26 (1.06)}$\\

\end{tabular}
\end{center}
\end{table}

\end{document}